\documentclass[12pt]{amsart}
\usepackage{amssymb}
\usepackage[T1]{fontenc}
\usepackage{graphicx}
\usepackage{constants}
\theoremstyle{plain}
\newtheorem{theorem}{Theorem}[section]
\newtheorem{lemma}[theorem]{Lemma}
\newtheorem{proposition}[theorem]{Proposition}
\newtheorem{corollary}[theorem]{Corollary}
\theoremstyle{definition}
\newtheorem{definition}[theorem]{Definition}
\newtheorem{remark}[theorem]{Remark}

\newtheorem*{thma}{Theorem A}
\newtheorem*{thmb}{Theorem B}

\numberwithin{equation}{section}

\newcommand{\R}{\mathbb{R}}
\newcommand{\Z}{\mathbb{Z}}

\newcommand{\CC}{\mathbb{C}}

\newcommand{\eps}{\varepsilon}

\renewcommand{\phi}{\varphi}

\newcommand{\caln}{\mathcal{N}}
\newcommand{\calp}{\mathcal{P}}
\newcommand{\ahat}{\hat{a}}
\newcommand{\atil}{\tilde{a}}
\newcommand{\qedwhite}{\hfill \ensuremath{\Box}}
\newconstantfamily{K-const}{
  symbol=K}
\newconstantfamily{Kappa-const}{
  symbol=\kappa}

\begin{document}

\title{Expansion properties of Double Standard Maps}
\author{Michael Benedicks}
\address{Matematiska institutionen, KTH, Lindstedtsv\"agen 25, S-100
  44 Stockholm, Sweden}
\email{michaelb@math.kth.se}
\author{Michal Misiurewicz}
\address{Department of Mathematical Sciences, IUPUI, 402 N. Blackford
  Street, Indianapolis, IN 46202, USA}
\email{mmisiure@math.iupui.edu}
\author{Ana Rodrigues}
\address{Department of Mathematics, Exeter University, Exeter EX4 4QF, UK}
\email{A.Rodrigues@exeter.ac.uk}
\thanks{The authors would like to thank the referee for his careful
  reading and several
  valuable suggestions and corrections. The authors would like to thank the
  G\"oran Gustafsson Foundation UU/KTH for financial support. Research
of Micha{\l} Misiurewicz was partially supported by grant number
4266012 from the Simons Foundation and research of Michael Benedicks was
partially supported by the Swedish Research Council, grant number 2016-05482}
\date{\today}

\begin{abstract}
For the family of Double Standard Maps $f_{a,b}=2x+a+\frac{b}{\pi}
\sin2\pi x \quad\pmod{1}$ we investigate the structure of the space of
parameters $a$ when $b=1$ and when $b\in[0,1)$. In the first case
the maps have a critical point, but for a set of parameters $E_1$ of
positive Lebesgue measure there is an invariant absolutely continuous
measure for $f_{a,1}$. In the second case there is an open nonempty
set $E_b$ of parameters for which the map $f_{a,b}$ is expanding. We
show that as $b\nearrow 1$, the set $E_b$ accumulates on many points
of $E_1$ in a regular way from the measure point of view.
\end{abstract}

\maketitle

\section{Introduction}\label{Introduction}

In one-dimensional dynamics, a lot is known about the families of
smooth maps with a critical point, like quadratic maps, and about the
maps that have no critical points (local diffeomorphisms of the
circle). Here we start to investigate what happens at the interface of
those two cases.

Consider the family of \emph{double standard maps} of the circle onto
itself, given by
\begin{equation}\label{DSM}
f_{a,b}=2x+a+\frac{b}{\pi} \sin2\pi x \quad\pmod{1},
\end{equation}
where the parameters $a,b$ are real, $a \in [0,1)$ and $b\in[0,1]$. In
fact, we consider $a$ from the circle $\R/\Z$, but since we are mostly
working locally (and far from $a=0$), considering $a$ real is simpler.
These family of maps were introduced in \cite{MR07}.

For $b=1$, maps of the family \eqref{DSM} have a unique cubic critical
point (at $c=1/2$) and negative Schwarzian derivative. Thus, they
behave similarly to the quadratic maps. In particular, there is a set
of parameters $a$ for which there is an invariant probability measure,
absolutely continuous with respect to the Lebesgue measure. For $b<1$,
there is no critical point, so the the maps are local diffeomorphisms.
Complexification of the maps, obtained by conjugacy via $e^{2\pi ix}$,
gives the family
\[
g_{a,b}(z)= e^{2\pi ia}z^2e^{b\left(z-\frac1z\right)}.
\]
Those maps are symmetric with respect to the unit circle, and factored
by this symmetry, they have only one critical point and no asymptotic values
in $\CC\setminus\{0\}$. Therefore a map $f_{a,b}$ has at most one
attracting or neutral periodic orbit (see~\cite{MR07, D10, FG07}).

One can also look at the family of double standard maps as a hybrid
between the family of \emph{standard maps}, studied by V.~Arnold
(see~\cite{A65}) and important in the creation of the KAM theory, and
\emph{expanding maps} of the circle (see~\cite{SS85}). Of course
instead of maps of degree 2 one can take maps of higher degrees and
the results will be practically the same (but we would introduce one
more parameter and loose a nice name of the family).

Some recent work has been done for classes of families that include
double standard maps. Misiurewicz and Rodrigues studied them
in~\cite{MR07, MR08}. Benedicks and Rodrigues~\cite{BR09} investigated
symbolic dynamics for this family. Universality for critical circle
covers was studied by Levin and \'Swi\c{a}tek,~\cite{LS2002}. Levin
and van Strien~\cite{LevvS2001} proved complex bounds, quasisymmetric
rigidity and density of hyperbolicity for a class of real analytic
maps which includes the double standard maps. Fagella and
Garijo~\cite{FG07} studied a class of complex maps containing the maps
$g_{a,b}$. Dezotti~\cite{D10} also considered maps $g_{a,b}$, and
using complex methods obtained important results on the real case.

As for the Arnold's family, for the double standard family we call the
sets for which there is an attracting periodic orbit of a given type
(period plus combinatorics) \emph{tongues}. Dezotti~\cite{D10} proved
that tongues are connected. The lowest tongue tip is at $b=1/2$, for
the period 1 tongue. If $0<b<1/2$, the map $f_{a,b}$ is expanding. At
higher $b$-levels there may be finitely or infinitely many tongues
(see~\cite{MR07}). In particular, at the critical level $b=1$ all
tongues are present, and it is easy to prove that they are dense at
this level (see~\cite{LevvS2001}). We show (in Theorem~A) that at the
lower levels $f_{a,b}$ can have an attracting or neutral periodic
orbit, and otherwise it is expanding. Moreover, the set of expanding
maps is dense in the complement of the tongues.

For simplicity, we will be using notation $f_a$ for $f_{a,1}$. A
parameter $a_0$ will be called an \emph{MT parameter} if the
trajectory of the critical point $c=1/2$ is preperiodic (but not
periodic).

In this case $f_{a_0}$has an absolutely continuous invariant measure,
\cite{Mis1}, and  it is also true that the critical value
$f_{a_0}(\frac{1}{2})$ satisfies the Collet-Eckmann condition,
i.e. that there is $C_{\text{\rm CE}}>0$ and  $\Cl[Kappa-const]{kappa:ce}>0$ such that for $a=a_0$ 
\begin{equation}\label{ce-est}
(f_{a}^n)'(f_{a}(c)) \geq C_{\text{\rm CE}}   e^{\Cr{kappa:ce} n},   \qquad \forall n\geq 0,
\end{equation}
which implies the existence of an absolutely continuous invariant
measure, \cite{NvS,vST}.

Using the methods of \cite{BC1} it is possible to prove

\begin{proposition}\label{prop:stretched:exp}
  There is a set of positive Lebesgue measure $\tilde{E}_1$ so that
  for all $a\in \tilde{E}_1$ there is $n_0(a)$ so that 
  \begin{equation}\label{stretched-ce-est}
(f_{a}^n)'(f_{a}(c)) \geq e^{n^{2/3} },   \qquad \forall n\geq n_0(a),
\end{equation}
Here $\frac{2}{3}$ can be replaced by any constant $\sigma<1$.
\end{proposition}

A parameter exclusion requiring 

\begin{equation}\label{mod-ba}
\text{dist}(f_{a}^j(c),c)\geq e^{- \sqrt{j}},\quad j \geq 1,
\end{equation}
will be sufficient to prove \eqref{stretched-ce-est} and then also  
Jacobson's theorem follows.

\smallskip
Using the methods of Large deviations of \cite{BC2} it is possible to
prove

\begin{proposition}\label{prop:CE:exp}

  There is a set of positive Lebesgue measure $E_1$ and some $\kappa>0$ so that
  for all $a\in E_1$

  \begin{equation}\label{ce-est-pos}
(f_{a}^n)'(f_{a}(c)) \geq C e^{\kappa n },   \qquad \forall n\geq 0.
\end{equation}
\end{proposition}

% Using the methods of \cite{BC1,BC2} is also possible to prove
% \eqref{ce-est} for $a\in E_1$, where $E_1$ is a set of postive Lebesgue
% measure, and using \cite{NvS,vST} Jacobson's theorem follows in this case.

For a similar result see \cite{TTY}.

In the present paper we will consider the non-critical case $0<b<1$
and use more elementary methods based on \cite{BC1}, which 
give stretched exponential growth of the type

\begin{equation}\label{ce-est-stretched}
(f_{a,b}^n)'(f_{a,b}(c)) \geq  e^{n^{2/3} },\qquad   n_0\leq n \leq \hat{N}(a,b),
\end{equation}
for all $a\in \tilde{E}_b$ for a set $\tilde{E}_b$, which is a finite
union of intervals. To obtain this it is sufficient to do parameter
exclusions of the type

\begin{equation}\label{mod-ba2}
\text{dist}(f_{a,b}^j(c),c)\geq \Cl{c:cba-sqrt}e^{- \sqrt{j}},\quad j \geq 1,
\end{equation}
and then prove exponential expansion in Section 8. The discussion of
the proof is more elaborated at the end of this section.

%%Here $\frac{2}{3}$ can be replaced by any $\sigma<1$.

% Note that this is not immediately enough to prove
% the existence of an absolutely continuous invariant measure.

% A parameter exclusion requiring \begin{equation}\label{mod-ba}
% \text{dist}(f_{a,b}^j(c),c)|\geq e^{- \sqrt{j}}\qquad 1\leq j\leq \hat{N}
% \end{equation}
% for some fixed $\hat{N}=\hat{N}(a,b)$ will be sufficent, to prove {\em
%   uniform expansion} (see Section 8), which in turn gives an existence
% of an absolute continuous invariant measure by classical results, see
% the discussion in the section below. 
% \medskip

We will outline the proof of Proposition \ref{prop:stretched:exp} after
the proof of Theorem A.

By the results of~\cite{BLvS03} and~\cite{BRLSvS08}, if $f_a(c)$
satisfies the Collet-Eckmann condition, then $f_a$ has an absolutely
continuous invariant measure. This is the analogue of Jakobson's
theorem~\cite{Ja81} in this case.

It is also possible to prove \eqref{ce-est} for $a$ in a set $E_1$
of positive Lebesgue measure, but with the present setup this would
require the method of Large deviations of \cite{BC2}, and this is not
required when  $0<b<1$.

Let us introduce some notations. For a fixed $b$, let us denote the
sets of those parameters $a$ for which $f_{a,b}$ has an attracting
(resp.\ neutral) orbit $T_b$ (resp.\ $TN_b$). Moreover, let $E_b$ be
the set of those parameters $a$ for which $f_{a,b}$ is
\emph{expanding}, that is, there exist $C>0$ and $\kappa>0$ such that
\begin{equation}\label{uniform}
(f_{a,b}^n)'(x)\geq C e^{\kappa n},   \qquad \forall n\geq 0\quad
\forall x\in{\mathbb T}.
\end{equation}
By the result of Ma\~n\'e~\cite{M85}, if $a$ does not belong to $T_b$
or $TN_b$, then it belongs to $E_b$. Observe that by the definition, a
small perturbation of an expanding map is also expanding, so $E_b$ is
open. In fact, the set $E=\{(a,b):a\in E_b, 0\le b<1\}$ is open in
$[0,1)\times[0,1)$.

Note that our definition of $E_1$ or $\tilde{E}_1$ is quite different
from the noncritical case, i.e. the case  of $E_b$ for $b<1$. Nevertheless,
there are some common features of the noncritical case, because if
$f_{a,b}$ is expanding, then by the results of Krzy\.zewski and
Szlenk~\cite{KS69}, or by the Lasota-Yorke Theorem~\cite{LY}, there
exists an absolutely continuous invariant measure.

Extending the methods of the proof of
\eqref{stretched-ce-est}, we prove the 
following theorem.

\begin{thma}\label{thma}
Let $a_0$ be a MT parameter for the family $\{f_a\}$. Denote
$\omega(\eps)=(a_0-\eps,a_0+\eps)$. Then for some $\eps_0>0$ there is
a function $b_0:(0,\eps_0) \to (0,1)$ such that
\[
\lim_{\eps\to 0}\frac{\inf\{|E_b\cap\omega(\eps)|:b\in(b_0(\varepsilon),1]\}}
{|\omega(\eps)|}=1.
\]
Here $|A|$ denotes the Lebesgue measure of the set $A$.
\end{thma}

This can be considered as the main result of the paper. It gives a
quantitative relation between the behavior of the system for $b<1$,
where the maps are local diffeomorphisms, and for $b=1$, the critical
case.

Finally, we prove a topological result, using very different methods.

\begin{thmb}\label{thmb}
For each $b<1$, the set $E_b$ is dense in the complement of $T_b$. In
particular, every interval of the parameters $a$ either is contained
in a closure of one tongue or intersects $E_b$.
\end{thmb}

The above theorem in some sense complements Theorem~A. Locally it says
less about the set $E_b$, but it applies to all $b<1$, not only to $b$
sufficiently close to $1$ (moreover, this closeness in Theorem~A
depends on $a_0$).

The paper is organized as follows. In Section~\ref{Partition} we
introduce notation and some definitions that we will use throughout
the paper. In Section~\ref{Outline} we prove the \emph{transversality
  condition} for maps of the family \eqref{DSM}. In
Section~\ref{Outside} we prove that we have exponential growth of the
derivative for an orbit of a map $f_{a,b}$ that moves outside an open
interval containing the the critical point when $(a,b)$ is a small
perturbation of an MT parameter $(a_0,1)$. In
Section~\ref{sec:bound-free-ess} we describe the induction including
its startup and prove
that the conditions on the Induction Statement are satisfied for the
first free return time. Furthermore,
we define the \emph{bound period} and prove some results concerning
the derivative growth during the bound period. In Section~\ref{gdl}
we prove the Global Distortion Lemma and in Section~\ref{slargedev2} we
start the proof of  Theorem~A. In
Section~\ref{improvedsel} we finish this proof. Finally, in
Section~\ref{pfthmc} we prove Theorem~B.

\smallskip
Let us indicate what is technically new in this paper compared to
previous work. The proof of 
Theorem A  is based on techniques in \cite{BC1} and \cite{BC2}.

\medskip
The main strategy of  the proof of Theorem A is the inductive proof of

  \begin{equation}\label{stretched-ce-est1}
(f_{a,b}^n)'(f_{a,b}(c)) \geq \Cl{c:celbstretched}e^{n^{2/3} },   \qquad 
\end{equation}
up to a certain time $\hat{N}$.

The set $\tilde{E}_1$ in the critical case $b=1$  is a Cantor set of positive measure
represented as
$$
E_1=\bigcap_{n=0}^\infty A_n,
$$
where each $A_n$ is a disjoint union of intervals $\{I_n^j\}$. Here
$A_{n+1}\subset A_n$ and the set $A_{n+1}$ are defined by removing subintervals  
of each $I_{n}^j$ according to two rules: First those subintervals
that do not satisfy an 
approach rate condition \eqref{mod-ba} for the critical point (or inflexion
point) $c=\frac12$ are deleted.  This replaces the basic assumption (BA) in
\cite{BC1} and \cite{BC2}. In the non critical  case $0<b<1$ this
condition corresponds to \eqref{mod-ba2}.
% $$
% \text{dist}(f_{a,b}^n(c),c)\geq C_{12}e^{-\sqrt{n}},\qquad \forall n\geq 1.
% $$

The proof of Theorem A is different from that of
Proposition \ref{prop:stretched:exp} in the critical case due to the 
fact that the behaviour at the inflexion point $c=\frac{1}{2}$ is
given by the Taylor series 

\begin{align*}
f_{a,b}(x)&=a+(2-2b)(x-\frac12)+\sum_{k=1}^\infty
a_{2k+1}(x-\frac{1}{2})^{2k+1}\\
&=a+(2-2b)(x-\frac12)+g(x).
\end{align*}

Here $g(x)/(x-\frac12)^3$ is bounded above and below by  strictly
numerical positive constants, i.e they depending only on the MT map
$f_{a_0}$. Similarly $g'(x)/(x-\frac{1}{2})^2$ has similar bounds from 
above and below.

This means that when a point $y=f_{a,b}^n(x)$ is close to $c=\frac12$,
the derivative $f_{a,b}'(y)$ will either be dominated by $2-2b$ or by
the quadratic term $g'(y)$ and these two cases will be treated differently.

The induction in the non-critical case $b<1$ can actually be
terminated at the time $\hat{N}$ defined by the condition that  constant
term in the derivative  
$2-2b$ can be of size comparable to $g'(x)$, which is quadratic, i.e.
\begin{equation}\label{stoppingtime}
2-2b\sim (e^{-\sqrt{\hat{N}}})^2,
\end{equation}
where $\sim$ means that the two sides are comparable within fixed
constants, which are only depending on $f_{a_0}$. Defining $\hat{N}$ this
way we can stop the induction at time $\hat{N}$ and the remaining 
set $E_b=\bigcap_{n=0}^{\hat{N}} A_n(b)$ is a finite union of intervals.

Another new aspect of the present paper is that for $b<1$ that the condition
of Collet-Eckmann type (again with $c=\frac12$) 
\begin{equation}
(f_{a,b}^n)'(f_{a,b}(c))\geq C e^{\kappa n}\qquad \forall n\geq 0\label{CE0},
\end{equation}
is no longer sufficient for the existence of absolutely continuous
invariant measures for $f_{a,b}$. There is however another argument
given in Section \ref{improvedsel} which 
uses \eqref{CE0} together with {\it bound period  estimates},
see Section \ref{Outside}, which prove uniform hyperbolicity, \eqref{uniform}.

\section{Notation}\label{Partition}

Throughout this paper, $C$ is a general numeric constant. For a set
$A\subset\R$ we will denote by $|A|$ its Lebesgue measure.

Consider the family of {\it double standard maps} given by~\eqref{DSM}
with $b=1$. Throughout this paper we denote $f_a(x)=f_{a,1}(x)$ and
$\xi_j(a)$ the orbit of the critical point: $\xi_j(a)=f_a^j(c)$.

For a general $b\leq 1$, we also use the notation
$\xi_j(a,b)=f_{a,b}^j(c)$. It is clear that when $b<1$, the point
$c=\frac12$ is an inflexion point. Sometimes we also use the notations
$f(x,a,b)$ for $f_{a,b}(x)$ and $f(x,a)$ for $f_a(x)$.

By $\partial_a f_{a,b}^j(x)$ we denote the partial derivative of
$f_{a,b}^j(x)$ with respect to $a$ and by $\partial_a f_a^j(x)$ the
partial derivative of $f_a^j(x)$ with respect to $a$.

\begin{definition}\label{DefnMT} A parameter $a=a_0$ will be called an
\emph{MT parameter}, if there is an integer $m$ and a period length
$\ell$ such that $\xi_m(a_0)$ is a periodic point of $f_{a_0}$ of
period $\ell$ and the multiplier
$\Lambda:=(f_{a_0}^\ell)'(\xi_m(a_0))$ is larger than $1$.
\end{definition}

As in \cite{BC1} and \cite{BC2}, we define a partition ${\mathcal
  Q}=\{ I_{r,l} \}$ of the {\it return interval}
$I^{*}=(c-\delta,c+\delta)$, where $\delta=e^{-r_\delta}$. We first divide $I^{*}$
\[
I^{*}=\bigcup_{r=r_\delta}^\infty I_r\cup\bigcup_{r=r_\delta}^\infty I_{-r},
\]
where $I_r$ is the interval $(c+e^{-r-1},c+e^{-r})$ for $r_\delta\leq
r <\infty$,
and $I_{-r}$ is the interval $(c-e^{-r},c-e^{-r-1})$.

We then subdivide $I_r$ into $r^2$ intervals of equal length with
disjoint interiors as follows
\[
I_r=\bigcup_{\ell=0}^{r^2-1}I_{r,\ell}.
\]
For convenience we also use the convention that $I_{r,r^2}=I_{r-1,0}$,
$r>0$, and the analogous convention for $r<0$.

Note that we have $| I_{r,l}|=\frac{e^{-r}}{r^2}(1-e^{-1})$ and
$|I_{r}|= e^{-r}(1-e^{-1})$. We will also need the extended interval
$$
I_{r,\ell}^+=I_{r,\ell-1}\cup I_{r,\ell} \cup I_{r,\ell+1}.
$$

For technical reasons we will also need a partition ${\mathcal Q}'=\{
I_{r,l} \}$, $|r|\geq r_\delta^1$, of an interval
$I^{**}=(c-\delta_1,c+\delta_1)$, where $|r|\geq r_\delta^1$, for some
$r_\delta^1<r_\delta$, i.e.\ $\delta_1=e^{-r_\delta^1}>\delta$.

A main tool in this paper is a sequence of partitions
${\mathcal P}_n$, $n=0,1,2,\dots$ of the parameter space which is
induced by the phase space partition. We define
\[
{\mathcal E}_n=\bigcup_{\omega\in {\mathcal P}_n}\omega.
\]

We call a time $n$ a {\it free return} if there is a parameter
interval $\omega$ belonging to a partition $\mathcal{P}_n$ such that
\[
\xi_n(\omega)=I_{r,\ell}.
\]
Similarly if we fix $b<1$, we will have
\[
\xi_n(\omega,b)=I_{r,\ell}.
\]
(In some cases these two conditions
for technical reasons will have to be replaced by  
$I_{r,\ell}\subset \xi_n(\omega)\subset I_{r,\ell}^+$ or
$I_{r,\ell}\subset \xi_n(\omega,b)\subset I_{r,\ell}^+$.)
\section{Transversality}\label{Outline}

In this section we prove the \emph{transversality condition} for maps
belonging to the family \eqref{DSM}.

\begin{lemma}\label{dera}
The following formula holds:
\begin{equation}\label{dera-eq}
\partial_a f^n_{a,b}(x)=\sum_{k=0}^{n-1}(f^k_{a,b})'(f^{n-k}_{a,b}(x))
=(f^{n-1}_{a,b})'(f_{a,b}(x))\sum_{k=0}^{n-1}
\frac1{(f^k_{a,b})'(f_{a,b}(x))}.
\end{equation}
\end{lemma}

\begin{proof}
We have
\begin{equation}\label{dera1}
\partial_a f^{n+1}_{a,b}(x)=1+f_{a,b}'(f^n_{a,b}(x))\cdot\partial_a
f^n_{a,b}(x)
\end{equation}
(note that $\partial_a f^0_{a,b}(x)=0$ and $\partial_a f^1_{a,b}(x)=1$).
Using this formula, we prove by induction
\begin{equation}\label{dera2}
\partial_a f^n_{a,b}(x)=\sum_{k=0}^{n-1}(f^k_{a,b})'(f^{n-k}_{a,b}(x)).
\end{equation}

If $n=0$, then both sides of~\eqref{dera2} are 0. Assume now
that~\eqref{dera2} holds for some $n$ and prove it for $n+1$ instead:
\begin{align*}
\partial_a f^{n+1}_{a,b}(x)&=1+f_{a,b}'(f^n_{a,b}(x))\cdot\sum_{k=0}^{n-1}
(f^k_{a,b})'(f^{n-k}_{a,b}(x))
=1+\sum_{k=0}^{n-1}(f^{k+1}_{a,b})'(f^{n-k}_{a,b}(x))\\
&=1+\sum_{k=1}^n(f^k_{a,b})'(f^{n-(k-1)}_{a,b}(x))=\sum_{k=0}^n
(f^k_{a,b})'(f^{(n+1)-k}_{a,b}(x)).
\end{align*}
Thus, by induction,~\eqref{dera2} holds for every $n$.

Now, we have
\[
(f^{n-k-1}_{a,b})'(f_{a,b}(x))\cdot
(f^k_{a,b})'(f^{n-k}_{a,b}(a))=(f^{n-1}_{a,b})'(f_{a,b}(x)).
\]
{}From this and~\eqref{dera2} we get
\begin{align*}
\partial_a f^n_{a,b}(x)&=(f^{n-1}_{a,b})'(f_{a,b}(x))\sum_{k=0}^{n-1}
\frac1{(f^{n-k-1}_{a,b})'(f_{a,b}(x))}\\
&=(f^{n-1}_{a,b})'(f_{a,b}(x))
\sum_{k=0}^{n-1}\frac1{(f^k_{a,b})'(f_{a,b}(x))}.
\end{align*}
\end{proof}

We get the following immediate corollary.

\begin{corollary}\label{dera_cor}
We have $\partial_a f_{a,b}(x)\equiv 1$ and if $n>0$ then $\partial_a
f^n_{a,b}(x)\ge 1$. Moreover,
\begin{equation}\label{qn-eq}
\frac{\partial_a\xi_n(a,b)}{(f_{a,b}^{n-1})'(f_{a,b}(c))}
=\sum_{k=0}^{n-1}\frac1{(f^k_{a,b})'(f_{a,b}(c))},
\end{equation}
so, in particular,
\begin{equation}\label{derax}
\partial_a\xi_n(a,b)\ge (f_{a,b}^{n-1})'(f_{a,b}(c)).
\end{equation}
In a special case, when there is a constant $\Cr{c:celbstretched}$ 
so that
\begin{equation}\label{eq:lower}
(f^\nu_{a,b})'(f_{a,b}(c))\geq \Cr{c:celbstretched}e^{\nu^{2/3}},\qquad
\nu=0,1,2,\dots,n-1
\end{equation}
then for all $\nu\leq n$ we obtain by combining the inequality  with 
the lower bound \eqref{derax}
\begin{equation}\label{deraceq}
1\leq\frac{\partial_a\xi_n(a,b)}{(f_{a,b}^{n-1})'(f_{a,b}(c))}\le q_*,
\end{equation}
where
\[
q_*=\sum_{i=0}^\infty \Cr{c:celbstretched}^{-1}e^{-i^{2/3}}.
\]
\end{corollary}

\begin{remark}We would like to emphasize the central role that
  Corollary \ref{dera_cor} plays in this paper. We prove the
  estimate \eqref{eq:lower} successively by induction on $\nu$. We can
  then conclude that \eqref{deraceq} holds with $n$ replaced by $\nu$  for a given
  $\nu$. From this estimate we conclude that the $x$- and
  $a$-derivative are comparable at a given time $\nu$. It is important
  that we prove the $x$-expansion first and then verify the
  comparison. The constant $q_*$ will be fixed, i.e. it only depends on 
$f_{a_0}$.
\end{remark}
We will also need the following lemma which can be viewed as a lower
bound for the Radon-Nikodym derivative of $\xi_\nu(a,b)\mapsto
\xi_\mu(a,b)$, $\nu<\mu$, (with respect to  $a\in\omega$).

\begin{lemma}\label{le:growth}
Suppose that $\omega$ is a parameter interval 
and $\nu<\mu$. Assume further that there is a constant $q'$ such that
for all $t\in\omega$
\begin{equation}\label{ax-dist-gen}
(f_{t,b}^{\nu-1})'(f_{t,b}(c))\ge\frac1{q'}\partial_a\xi_\nu(t,b).
\end{equation}
Then
\[
|\xi_{\mu}(\omega,b)|\geq \frac1{q'}\inf_{a\in\omega}
(f_{a,b}^{\mu-\nu})'(f_{a,b}^\nu(c))\cdot|\xi_\nu(\omega,b)|.
\]

\end{lemma}

\begin{proof}
By Corollary~\ref{dera_cor},
\begin{equation}\label{deco1}
|\xi_{\mu}(\omega,b)|=\int_\omega \partial_a\xi_{\mu}(t,b)\,dt\ge
\int_\omega (f_{t,b}^{\mu-1})'(f_{t,b}(c))\,dt.
\end{equation}
However, by~\eqref{ax-dist-gen} we have
\[
(f_{t,b}^{\mu-1})'(f_{t,b}(c))=(f_{t,b}^{\mu-\nu})'(f_{t,b}^\nu(c))
\cdot (f_{t,b}^{\nu-1})'(f_{t,b}(c))\ge\inf_{a\in\omega}
(f_{a,b}^{\mu-\nu})'(f_{a,b}^\nu(c))\cdot
\frac1{q'}\partial_a\xi_\nu(t,b).
\]
Together with~\eqref{deco1}, we get
\[
|\xi_{\mu}(\omega,b)|\ge \frac1{q'}\inf_{a\in\omega}
(f_{a,b}^{\mu-\nu})'(f_{a,b}^\nu(c))\cdot
\int_\omega\partial_a\xi_\nu(t,b)\,dt= \frac1{q'}\inf_{a\in\omega}
(f_{a,b}^{\mu-\nu})'(f_{a,b}^\nu(c))\cdot|\xi_\nu(\omega,b)|.
\]
\end{proof}

\section{The outside expansion}\label{Outside}

The aim of this section is to prove that we have exponential growth of
the derivative for an orbit of a map $f_{a,b}$ that moves outside an
open interval $I$ containing $c$, when $(a,b)$ is a small perturbation
of an MT parameter $(a_0,1)$. We consider the parameter space
$\R/\Z\times(0,1]$, and when we speak of a neighborhood of $(a_0,1)$,
we mean its neighborhood in this space.

By $|x-y|$ we denote the distance between $x$ and $y$ on the circle.
Since the points $x$ and $y$ will be usually close to each other,
this makes perfect sense. Denote
\begin{equation}\label{distcrit}
\overline{d}=\min_{j\geq 1}|c-f_{a_0}^j(c)|.
\end{equation}
By the definition of an MT parameter, we have $\overline{d}>0$.

Since $f_{a_0}$ has negative Schwarzian derivative, the following lemma
follows immediately from Theorem~1.3 of~\cite{Mis1} (in a general case
one could use also the result of Ma\~n\'e (see \cite{M85})).

\begin{lemma}\label{m0}
Let $I$ be an open interval containing $c$. Then there exists a
neighborhood $\caln$ of $(a_0,1)$, positive constants $\Cl{c:ce},\Cl[Kappa-const]{kappa:mane}$,
and an integer $M_1$ such that if $(a,b)\in\caln$ then
\begin{itemize}
\item[(i)] if $x, f_{a,b} (x),\dots, f_{a,b}^{n-1}(x) \notin I$, then
\[
(f_{a,b}^n)'(x)> \Cr{c:ce} e^{\Cr{kappa:mane} n};
\]
\item[(ii)] if $x,f_{a,b}(x) \dots,\ f_{a,b}^{n-1}(x) \notin I$ and
$n\ge M_1$, then
\[
(f_{a,b}^n)'(x)> e^{\Cr{kappa:mane} n}.
\]
\end{itemize}
\end{lemma}

\begin{proof}
By Theorem~1.3 of~\cite{Mis1} (or a result of Ma\~n\'e \cite{M85}),
there exists $L>0$ and $\Cr{kappa:mane}'>0$ such that if $x, f_{a_0}
(x),\dots, f_{a_0}^{L-1}(x) \notin I$, then $(f_{a_0}^L)'(x)>
e^{\Cr{kappa:mane}' L}$. Therefore, if $\caln$ is a sufficiently small
neighborhood of $(a_0,1)$, then for all $(a,b)\in\caln$ and $x$ such
that $x, f_{a,b} (x),\dots, f_{a,b}^{L-1}(x) \notin I$, we have
$(f_{a,b}^L)'(x)> e^{\Cr{kappa:mane}' L}$. Since the infimum of
$(f_{a,b}^i)'(x)$ over $(a,b)\in\caln$, $x\notin I$ and
$i=0,1,\dots,L-1$ is positive, there exists a positive constant $\Cr{c:ce}$
such that~(i) holds with $\Cr{kappa:mane}=\Cr{kappa:mane}'$. Thus it also holds
with ${\Cr{kappa:mane}}=\Cr{kappa:mane}'/2$, and then~(ii) holds with any
  \[
  M_1>\frac{-\log\Cr{c:ce}}{\Cr{kappa:mane}}.
\]
\end{proof}

Now we fix a positive constant $\beta>0$. It will depend only on the
unperturbed map $f_{a_0}$ and can be chosen as, say
$\frac{1}{100}\min(\tilde{\kappa},\Cr{kappa:mane1})$. Here 
$\tilde{\kappa}=(1/\ell)\log\Lambda$, where $\Lambda$ is the
multiplier of the repelling periodic point of the MT-point, and
$\Cr{kappa:mane1}$ is the, Lyaponov exponent in Lemma \ref{lfree}. 

%Later we will specify how
%small it should be. It can be 

\begin{definition}\label{BoundMT}
Let $x \in I^{**}=(c-\delta_1,c+\delta_1)$. We say that $x$ is
\emph{$\beta$-bound to the critical point $c$ up to time $p$  for $f_{a,b}$},
if $p$ is the maximal integer such that
\begin{equation}
|f_{a,b}^j(x)-f_{a,b}^j(c)|\leq e^{-\beta j},\qquad \forall j \leq p.
\end{equation}
\end{definition}

Observe that for every $a,b$ (where $b\le 1$) and every $x$ we have
\begin{equation}\label{derxbd}
f_{a,b}'(x)\le 4 \quad\text{and}\quad |f_{a,b}''(x)|\le 4\pi<13.
\end{equation}

Let us state a version of the Bound Distortion Lemma (see \cite{BC1}
and \cite{BC2}).

\begin{lemma}\label{bounddistMT}
If $\delta_1$ is sufficiently small, then there is a constant
$\Cl{c:bp}(\delta_1)>1$, which converges to $1$ as $\delta_1\to 0$, such
that for every $x \in I^{**}=(c-\delta_1,c+\delta_1)$ if $x$ is $\beta$-bound
to $c$ up to time $p$ for $f_{a_0}$, then
\begin{equation}\label{estimate1MT}
\frac{1}{\Cr{c:bp}}<\frac{(f_{a_0}^k)'(f_{a_0}(x))}
{(f_{a_0}^k)'(f_{a_0}(c))}<{\Cr{c:bp}}
\end{equation}
for all $k \leq p$. Moreover, there is a constant $\Cl{c:ubp}=\Cr{c:ubp}(\delta_1)
>0$, which converges to $0$ as $\delta_1\to 0$, such that
\begin{equation}\label{e2}
|f_{a_0}^k(x)-f_{a_0}^k(c)|<\Cr{c:ubp}
\end{equation}
for all $k \leq p$.
\end{lemma}

\begin{proof}
Assume that $x$ is bound to $c$ up to time $p$. Now choose 
$p_1=\frac{1}{10}\log(1/\delta_1)$.

Then by~\eqref{derxbd} we can estimate
\[
|f_{a_0}^j(x)-f_{a_0}^j(c)|\le \delta_1 4^j\le \delta_1 4^{p_1}
\]
if $j\le p_1$ and
\[
|f_{a_0}^j(x)-f_{a_0}^j(c)|\le e^{-\beta j}\le e^{-\beta p_1}
\]
if $p_1<j\le p$. Thus, if $\delta_1$ is sufficiently small then
$|f_{a_0}^j(x)-f_{a_0}^j(c)|\le {\overline d}/2$ and therefore
\begin{equation}\label{m0a}
|f_{a_0}^j(x)-c|\ge \frac{\overline d}2
\end{equation}
for all $j\le p$.

 % and
% \begin{equation}\label{m0b}
% \sum_{j=1}^k |f_{a_0}^j(x)-f_{a_0}^j(c)|\to 0\qquad \text{as}\qquad
% \delta_1\to 0
% \end{equation}
This also proves the last statement of the lemma.

We have
\begin{align}
\begin{split}  
\frac{(f_{a_0}^k)'(f_{a_0}(x))}{(f_{a_0}^k)'(f_{a_0}(c))}
&=\prod_{j=1}^k \left(1+\frac{f_{a_0}'(f_{a_0}^j(x))-
f_{a_0}'(f_{a_0}^j(c))}{f_{a_0}'(f_{a_0}^j(c))}\right)\\
&\le\exp\left(\sum_{j=1}^k\frac{13|f_{a_0}^j(x)-f_{a_0}^j(c)|}
  {f_{a_0}'(f_{a_0}^j(c))}\right)\\
&\le\exp\left\{\Cl[K-const]{k:kboundp}\left(\delta_1p_14^{p_1}+\sum_{j=p_1+1}^ke^{-\beta
      j}\right)\right\}\label{m0c}.
\end{split}
\end{align}

The last sum in the exponential may be empty.

% We have used that ince $\overline{d}>0$, the denominators
% $f_{a_0}'(f_{a_0}^j(c)$ above are bounded away from 0. 

% \begin{equation}\label{m0c}
% \frac{(f_{a_0}^k)'(f_{a_0}(x))}{(f_{a_0}^k)'(f_{a_0}(c))}\le
% \exp\left(K_3\sum_{j=1}^k|f_{a_0}^j(x)-f_{a_0}^j(c)|\right).
% \end{equation}

Similarly, using~\eqref{m0a}, we get
\begin{equation}\label{m0d}
\frac{(f_{a_0}^k)'(f_{a_0}(c))}{(f_{a_0}^k)'(f_{a_0}(x))}\le
\exp\left\{\Cl[K-const]{k:K4}\left(\delta_1p_14^{p_1}+\sum_{j=p_1+1}^ke^{-\beta j}\right)\right\}\
%%\exp\left(K_4\sum_{j=1}^k|f_{a_0}^j(x)-f_{a_0}^j(c)|\right).
\end{equation}

The sums in~\eqref{m0c} and~\eqref{m0d} are
bounded by a constant, which only depends on $\delta_1$, so we
get~\eqref{estimate1MT}. Moreover, by~\eqref{m0c} and \eqref{m0d}, $\Cr{c:bp}>1$
converges to $1$ as $\delta_1\to 0$.
\end{proof}
% \begin{remark}\label{bounddistMTrem}

%   Lemma \ref{bounddistMT} will also hold in the more
%   general  case when it is assumed that  $f_{a,b}(x)$, $f_{a,b}(x)$,
%   \dots $f_{a,b}^k(x)$,   $x\in I^{**}$, is bound to $f_{a_0}(c)$,
%   \dots $f_{a_0}^k(c)$.  Since the estimates
%   are the same the same proof works.

% \end{remark}
  
Set $\tilde{\kappa}=(1/\ell)\log\Lambda$. Then there is a constant $C_6=C_6(a_0)$
so that 
\begin{equation}\label{estBP1}
(f_{a_0}^j)'(f_{a_0}(c))\geq \Cl{c:lfp} e^{\tilde{\kappa} j}
\end{equation}
for all $j\ge 1$.

At $c$, the first two derivatives of $f_a$ vanish, but the third one
does not. Therefore, there are positive constants $\Cl{c:lcr},\Cl{c:ucr}$ such that
for all $a$ sufficiently close to $a_0$
\begin{equation}\label{third}
\Cr{c:lcr}|x-c|^3<|f_a(x)-f_a(c)|<\Cr{c:ucr} |x-c|^3
\end{equation}
whenever $x\in I^{**}$. 
If $\delta_1$ is small, the constants $\Cr{c:lcr}$
and $\Cr{c:ucr}$ can be made close to each other. Similarly, for some
positive constants $\Cl{c:lder}',\Cl{c:uder}'$,
\begin{equation}\label{second}
\Cr{c:lder}'(x-c)^2< f_a'(x)<\Cr{c:uder}' (x-c)^2
\end{equation}
whenever $x\in I^{**}$. If instead of $f_a$ we consider $f_{a,b}$ with
$b$ sufficiently close to 1, we similarly obtain
\begin{equation}\label{thirdb}
|x-c|\left(2-2b+\Cr{c:lcr}(x-c)^2\right)<|f_{a,b}(x)-f_{a,b}(c)|<|x-c|\left(2-2b+\Cr{c:ucr}(x-c)^2\right) 
\end{equation}
and
\begin{equation}\label{secondb}
2-2b+\Cr{c:lder}(x-c)^2< f_{a,b}'(x)< 2-2b+\Cr{c:uder}(x-c)^2,
\end{equation}
and we chose $\Cr{c:lder}$ and  $\Cr{c:uder}$ so that these estimates are valid for 
all $b\leq 1$.
Moreover, we have
\begin{equation}\label{sec_der}
|f_{a,b}''(x)|\le 8\pi^2|x-c|<80|x-c|.
\end{equation}

In the following lemma we estimate the length of the bound period.

\begin{lemma}\label{le:ce}
If $\delta_1$ is sufficiently small, $x$ is $\beta$-bound to $c$ up to time
$p$ for $f_{a_0}$, and $x\in I^{**}\setminus\{c\}$, then
\begin{equation}\label{estp1}
p<-\frac{4}{\tilde{\kappa}}\log|x-c|.
\end{equation}
In the particular case when $x\in I_{\pm r}$ we obtain
\begin{equation}\label{eq-p-est-1}
p\leq \frac{4r}{\tilde{\kappa}}.
\end{equation}
\end{lemma}

\begin{proof}
By Lemma~\ref{bounddistMT}, we have
\[
|f_{a_0}^p(x)-f_{a_0}^p(c)|>\frac{(f_{a_0}^p)'(f_{a_0}(c))}{\Cr{c:bp}}
|f_{a_0}(x)-f_{a_0}(c)|.
\]
Taking into account~\eqref{estBP1} and~\eqref{third}, we get
\[
1>|f_{a_0}^p(x)-f_{a_0}^p(c)|>\frac{\Cr{c:lfp}\Cr{c:lcr}}{\Cr{c:bp}} e^{\tilde{\kappa} p}
|x-c|^3.
\]
If $\delta_1$ is small, then $\Cr{c:bp}<2$, so taking logarithms gives us
\[
\log\frac{\Cr{c:lfp} \Cr{c:lcr}}2+\tilde{\kappa} p+3\log|x-c|<0.
\]
If $\delta_1$ is small, then $-\log|x-c|$ is large, so we get
$-\log(\Cr{c:lfp} \Cr{c:lcr}/2)<-\log|x-c|$, and therefore $\tilde{\kappa} p<-4\log|x-c|$.
\end{proof}

We need a derivative estimate for an orbit of $f_{a_0}$ that
moves completely outside $I^*=(c-\delta,c+\delta)$ or
$I^{**}=(c-\delta_1,c+\delta_1)$ but returns to one of these intervals
at time $n$.

In the proof of the next lemma we will use the fact that $f_a$ has
negative Schwarzian derivative. This result can be generalized to the
$C^2$ case (see van Strien \cite{vST}).

\begin{lemma}\label{le:lowerboundret}
Let $\overline{d}$ be as in \eqref{distcrit}. For every $\delta_1 \in
(0, \overline{d}/2)$ and for every $n \geq 1$, if $x$ is such that
$f^j_{a_{0}}(x) \notin I^{**}$ for $j=0,\dots,n-1$, and $
f^n_{a_0}(x)\in I^{**}$, then $(f^n_{a_0})'(x)> \overline{d} /2.$
\end{lemma}

\begin{proof}
On each side of $x$ there are the two closest preimages of $c$ of order less
than $n$: $\eta_1<x$ and $\eta_2>x$. Then $f_{a_0}^n$ has positive
derivative on $(\eta_1,\eta_2)$ and has negative Schwarzian derivative
on that interval. Therefore on one of the intervals $[\eta_1,x]$ and
$[x,\eta_2]$ the maximum of the derivative $(f_{a_0}^n)'$ is attained
at $x$. We may assume that this is the interval $[\eta_1,x]$. Then
$f_{a_0}^n(\eta_1)=f_{a_0}^k(c)$ for some $k>0$, so
\[
|f_{a_0}^n(\eta_1)-f_{a_0}^n(x)|\ge\overline{d} -\delta_1 >
\overline{d} /2.
\]
By the Mean Value Theorem,
\[
|f_{a_0}^n(\eta_1)-f_{a_0}^n(x)|=(f_{a_0}^n)'(t) |\eta_1-x|
\le (f_{a_0}^n)'(t)
\]
for some $t\in(\eta_1,x)$, and thus,
\[
(f^n_{a_0})'(x)\ge (f_{a_0}^n)'(t)> \overline{d} /2.
\]
\end{proof}

%From now on we will fix $\delta_1<\overline{d}/2$, say
% $\delta_1=\overline{d}/4$. It is important that this $\delta_1$ only
% depends on the unperturbed map $f_{a_0}$.

\medskip
%%The constant $\delta_1$ will be fixed so the conditions
%%$\delta_1<\frac{\overline{d}}{2}$ and  \eqref{eq45}
%below are satisfied.

\medskip
In the following lemma we consider what we call a \emph{free period}.

\begin{lemma}\label{lfree}
Given $\delta_1$ sufficiently small, there is a neighborhood $\caln$ of $(a_0,1)$ in the parameter space
and positive constants $C^*$ and $\Cl[Kappa-const]{kappa:mane1}$, such that, if $(a,b)\in{\mathcal N}$ then if
$x,f_{a,b}(x),\dots, f_{a,b}^{q-1}(x)\not\in I^{**}$ and
$f_{a,b}^q(x)\in I^{**}$ then
\begin{equation}\label{estfree}
(f_{a,b}^q)'(x)\geq C^*e^{\Cr{kappa:mane1} q}.
\end{equation}
Here the constant $C^*$ depends only on the unperturbed map $f_{a_0}$,
while $\Cr{kappa:mane1}$  depends on $\delta_1$.
\end{lemma}

\begin{proof}
By Lemma~\ref{m0}, $(f_{a,b}^{M_1})'(x)\ge e^{{\Cr{kappa:mane}}
  M_1}$. For a general $q$ write $q=kM_1+\ell$, $0\leq \ell <M_1$. Choose
$\Cr{kappa:mane1}'$ so that
$e^{\Cr{kappa:mane1}'M_1}\leq 2$.

Since $\ell<M_1$, then by Lemma~\ref{le:lowerboundret},
$(f_{a,b}^{\ell})'(f_{a,b}^{kM_1}(x))> \overline{d}/4$ (here we can extend the estimate to
a neighborhood of ${\mathcal N}_2$ of  $(a_0,1)$ because we consider
only finitely many iterates of the map). Then for $(a,b)\in {\mathcal
  N}= {\mathcal N}_1\cap {\mathcal N}_2$
\begin{align*}
(f_{a,b}^{q})'(x)&=(f_{a,b}^{kM_1})'(x)(f_{a.b}^\ell)'(f_{a,b}^{kM_1}(x))\geq e^{\kappa_2
             kM_1}\cdot \frac{\overline{d}}{4}\\
  &\geq  e^{\kappa_2kM_1}\frac{\overline{d}}{8}\cdot
    e^{\Cr{kappa:mane1}'M_1}
    \geq \frac{\overline{d}}{8} e^{\kappa_2kM_1+\Cr{kappa:mane1}'\ell}
\end{align*}
so \eqref{estfree} follows with $C^*=\overline{d}/8$
and $\Cr{kappa:mane1}=\min(\kappa_2,\Cr{kappa:mane1}')$. Note that
as required, $C^*$ only depends on the unperturbed map $f_{a_0}$, while
$\Cr{kappa:mane1}$ depends on $\delta_1$.

\end{proof}

% {\it Remark.}Note that $M_1$ depends on $\delta_1$ but in its turn
% $\delta_1$ is chosen depending only on the unperturbed map
% $f_{a_0}$. Hence  $C^*$ and $\Cr{kappa:mane1}$ are constructed so they
% only depend on the unperturbed map. 

\medskip
We will also need an estimate of the derivative during the bound
period.

\begin{lemma}\label{dercompBP}
Assume that $9\beta\le\tilde{\kappa}$. Let $C^*$ be the constant from
Lemma~\ref{lfree}. Then there is an arbitrarily small $\delta_1$ such that if $x$ is $\beta$-bound
to $c$ up to time $p$ for $f_{a_0}$ and $x\in
I^{**}=(c-\delta_1,c+\delta_1)$ then
\begin{equation}\label{pbounda0}
(f_{a_0}^p)'(x)>\frac1{C^*}e^{\frac{\tilde{\kappa}}{4}p}.
\end{equation}
\end{lemma}

\begin{proof}
By the Mean Value Theorem, there is an $y$ between $f_{a_0}(x)$ and
$f_{a_0}(c)$, such that
\[
|f_{a_0}^p(x)-f_{a_0}^p(c)|=(f_{a_0}^{p-1})'(y)|f_{a_0}(x)-f_{a_0}(c)|.
\]
By this, Lemma~\ref{bounddistMT} and~\eqref{third}, there exists a
constant $\Cl[K-const]{k:x3-term}$ such that if $\delta_1$ is sufficiently small then
\[
|f_{a_0}^p(x)-f_{a_0}^p(c)|<\Cr{k:x3-term}e^{\tilde{\kappa} p}|x-c|^3.
\]
Similarly, since $(f_{a_0}^p)'(x)=f_{a_0}'(x)\cdot(f_{a_0}^{p-1})'
(f_{a_0}(x))$, we get by Lemma~\ref{bounddistMT} and~\eqref{second}
that there exists a constant $\Cl[K-const]{k:lowerbp}$ such that if $\delta_1$ is
sufficiently small then
\[
(f_{a_0}^p)'(x)>\Cr{k:lowerbp}e^{\tilde{\kappa} p}|x-c|^2.
\]

By the definition of $p$ we have
\[
|f_{a_0}^{p+1}(x)-f_{a_0}^{p+1}(c)|>e^{-\beta (p+1)},
\]
and therefore for some constant $\Cl[K-const]{k:lowerbp1}$
\[
|f_{a_0}^p(x)-f_{a_0}^p(c)|>\Cr{k:lowerbp1}e^{-\beta p}.
\]
Thus,
\[
\Cr{k:x3-term}e^{\tilde{\kappa} p}|x-c|^3>\Cr{k:lowerbp1}e^{-\beta p},
\]
so
\[
|x-c|^2>\Cr{k:lowerbp1}^{2/3}\Cr{k:x3-term}^{-2/3}e^{(2/3)(-\beta-\tilde{\kappa})p}.
\]
Together with an earlier estimate, this gives us
\[
(f_{a_0}^p)'(x)>\Cr{k:lowerbp}e^{\tilde{\kappa} p}\Cr{k:lowerbp1}^{2/3}\Cr{k:x3-term}^{-2/3}
e^{(2/3)(-\beta-\tilde{\kappa})p} =\Cr{k:lowerbp}\Cr{k:lowerbp1}^{2/3}\Cr{k:x3-term}^{-2/3}
e^{(1/3)(\tilde{\kappa}-2\beta)p}.
\]
Since $9\beta\le\tilde{\kappa}$, we have
\[
\frac13(\tilde{\kappa}-2\beta)>\frac7{27}\tilde{\kappa},
\]
and therefore~\eqref{pbounda0} holds if
\begin{equation}\label{eq45}
C^*>\Cr{k:lowerbp}^{-1}\Cr{k:lowerbp1}^{-2/3}\Cr{k:x3-term}^{2/3}e^{-\frac{\tilde{\kappa}}{108}p(\delta_1)},
\end{equation}
where $p(\delta_1)$ is the bound period associated with $\delta_1$.
Since $p(\delta_1)\to\infty$ as $\delta_1\to 0$, and $C^*$ is
independent of $\delta_1$, the above inequality holds if
$\delta_1$ is sufficiently small.
\end{proof}

\begin{remark}\label{rem45}
If in~\eqref{pbounda0} we replace (on both sides of the inequality) $p$
by $p-1$ or $p+1$, then at the right-hand side of in~\eqref{eq45}
there will be one more multiplicative constant. Since it was irrelevant
in the proof what constant is there, Lemma~\ref{dercompBP} still holds
with a suitably modified inequality~\eqref{pbounda0}.
\end{remark}

The following lemma is very similar to Lemma~\ref{m0}, but the
exponent in the estimate does not depend on the size of the
neighborhood of $c$ that we consider. In this lemma we
assume that $\delta_1$ is sufficiently small (so that the lemmas that
we use in the proof hold) but fixed.

\begin{lemma}\label{mane}
Let $I$ be an open symmetric interval around $c$, whose closure is contained
in $I^{**}$. Fix a sufficiently small neighborhood $\caln$ of
$(a_0,1)$ (depending on $I$). Then there are constants $\Cl{c:celb}>0$ and
$\Cl[Kappa-const]{kappa:outside}>0$,  (independent of $I$) and an integer $M$ (depending on
$I$) such that for $(a,b)\in\caln$
\begin{itemize}
\item[(i)] if $x,f_{a,b}(x),\dots, f_{a,b}^{n-1}(x) \notin I$ and
$f_{a,b}^n(x) \in I^{**}$, then
\[
(f_{a,b}^n)'(x)\geq \Cr{c:celb} e^{\Cr{kappa:outside} n};
\]
\item[(ii)] if $x,f_{a,b}(x),\dots, f_{a,b}^{n-1}(x) \notin I$ and
$n\geq M$, then
\[
(f_{a,b}^n)'(x)\geq e^{\Cr{kappa:outside} n}.
\]
\end{itemize}
\end{lemma}
{\it Remark.} Note that we state (i)   with the weaker assumption
$f_{a,b}^n(x) \in I^{**}$ instead of the more natural $f_{a,b}^n(x)
\in I$. This slightly stronger statement will be used in the proof of (ii).

\begin{proof}  Let $0=t_0<t_1<t_2<\dots<t_S<t_{S+1}=n$, where $t_i$ for
$i\in\{1,2,\dots,S\}$ are the times when $f_{a,b}^{t_i}(x)\in
I^{**}\setminus I$. We want to estimate
\[
(f_{a,b}^n)'(x)=\prod_{j=0}^S (f_{a,b}^{t_{j+1}-t_j})'(f_{a,b}^{t_j}(x)).
\]
The times from $[t_0,t_1)$ form a free period; let $t_1-t_0=q_0$.
Hence, by Lemma~\ref{lfree},
\[
(f_{a,b}^{q_0})'(x)\geq C^*e^{\Cr{kappa:mane1} q_0}.
\]
Consider now times from $[t_j,t_{j+1})$, where $j>0$. We can write it
as a union of a bound period $[t_j,t_j+p_j)$  and a free period
$[t_j+p_j,t_{j+1})$, and we write its length as $t_{j+1}-t_j=p_j+q_j$.
For the bound periods $[t_j,t_j+p_j)$ we can use the estimate from
Lemma~\ref{dercompBP} if $\caln$ is sufficiently small, because by
Lemma~\ref{le:ce} we work only with the finite number of iterates (for
$I$ fixed; this is why $\caln$ depends on $I$). Although $p$
 may depend on the map that we are using, if $\caln$ is sufficiently
 small, it only may change to $p\pm 1$, and then by Remark~\ref{rem45}
 we can still use Lemma~\ref{dercompBP}.

Thus, for the bound periods $[t_j,t_j+p_j)$ we get
\begin{equation}\label{eqbound}
(f_{a,b}^{p_j})'(f_{a,b}^{t_j}(x))\geq \frac1{C^*}e^{\frac{\tilde{\kappa}}{4} p_j}
\end{equation}

and for the free period, as before, the estimate from
Lemma~\ref{lfree} gives us
\[
(f_{a,b}^{q_j})'(f_{a,b}^{t_j+p_j}(x))\geq C^*e^{\Cr{kappa:mane1} q_j}.
\]
Combining these estimates we get
\[
(f_{a,b}^n)'(x)\geq C^*e^{\Cr{kappa:mane1}q_0}\prod_{j=1}^S\frac1{C^*}
e^{\frac{\tilde{\kappa}}{4} p_j }\cdot C^*e^{\Cr{kappa:mane1} q_j}\geq C^*e^{\Cr{kappa:outside}'
n},
\]
with $\Cr{kappa:outside}'=\min(\Cr{kappa:mane1}, \tilde{\kappa}/4)$. This completes
the proof of (i).
\end{proof}

Under the assumptions of (ii) instead of (i) we make the same
construction and estimates. The only difference is that we do not know
that $f_{a,b}^n(x)\in I^{**}$, so we lose information about the last
period. There are two cases.

{\it Case 1.} $f_{a,b}^n(x)$ is still in bound state to the last
return to $I^{**}$  at time $t_S$. At time $t_S$ we can use the
estimate of (i)
$$
(f^{t_S}_{a,b})'(x)\geq C_{11}e^{\Cr{kappa:outside}'t_S}.
$$
The derivative contribution at time $t_S$ is
$$
f_{a,b}'(f^{t_S}_{a,b}(x))\geq \Cr{c:lder}'\delta^2
$$
Then there is a derivative contribution from the time
$[t_{S}+1,t_S+j]$, $j=n-t_S$. Since $1\leq j\leq p_S$ we can use 
the Collet-Eckmann condition \eqref{ce-est} and the distorsion estimate
Lemma \ref{bounddistMT}, combined with continuity in $a$ for  $a\in
{\mathcal A}$, and the fact that $p_S$ is bounded to conclude that, say
$$
(f_{a,b}^j)'(f_{a,b}^{t_{S}+1}(x))\geq \frac{1}{2} C_{\text{\rm
    CE}}e^{\Cr{kappa:ce} j}.
$$
Combining these estimates and using the the chain rule we get  that
$$
(f^n_{a,b})'(x)\geq\Cr{c:celb} e^{\Cr{kappa:outside} t_S}\cdot
\Cr{c:lder}'\delta^2\cdot
 \frac{1}{2}C_{\text{\rm CE}}
e^{\frac{\tilde{\kappa}}{4}j} 
$$
and since $n\geq M$, where $M$  is allowed to depend on $\delta$ this
gives the estimate (ii) with a suitable $ \Cr{kappa:outside}''<\Cr{kappa:outside}'$.

{\it Case 2.} $n\geq t_S+p_S$. In this case we can use \eqref{eqbound}
with $j=S$ to obtain
\begin{equation}
(f_{a,b}^{p_S})'(f_{a,b}^{t_j}(x))\geq \frac1{C^*}e^{\frac{\tilde{\kappa}}{4} p_S}
\end{equation}
%and the trivial
%%estimate $f_{a,n}'(x)\geq \Cr{c:lder}'\delta_1^2$, $x\not
%%\in(c-\delat,c+\delta)$ to conclude that 

%\begin{equation}
%(f_{a,b}^{p_S})'(f_{a,b}^{t_j}(x))\geq \frac1{C^*}e^{\frac{\tilde{\kappa}}{4} p_S}
%\end{equation}
Then

$$
(f^n_{a,b})'(x)\geq  (f^{t_S}_{a,b})'(x)(f^{p_S}_{a,b})'(f^{t_S}_{a,b}(x)) ((f^{q_S}_{a,b})'(f^{t_S+p_S}_{a,b}(x)),
$$
where $q_S=n-(t_S+p_S)$. We get using Lemma \ref{m0} and  the simple
estimate $(f_{a,b})'(x)\geq \Cr{c:lder}'\delta_1^2$ for $|x-c|\geq
\delta_1$ that  
$$
  (f^{q_S}_{a,b})'(f^{t_S+p_S}_{a,b}(x))\geq 
  \begin{cases} e^{\Cr{kappa:mane}q},\qquad\qquad\ q_S\geq M_1\\
    (\Cr{c:lder}'\delta_1^2)^{q_S},\qquad q_S< M_1.
\end{cases}
$$
Using that the constants $\delta_1$, $\Cr{c:lder}'$,
$\Cr{kappa:mane}$ and $M_1$
only depend on $f_{a_0}$, we conclude that (ii) holds with $\Cr{kappa:outside}=\Cr{kappa:outside}'''$ for $n\geq M$, if
$M$ is sufficienty large. The final $\Cr{kappa:outside}$ is then
chosen as $\Cr{kappa:outside}=\min(\Cr{kappa:outside}',\Cr{kappa:outside}'',\Cr{kappa:outside}''')$.

% If the last period is free, then we usually cannot use the estimate
% from Lemma~\ref{lfree}. Instead, we use the estimate from
% Lemma~\ref{m0} (with $I=I^{**}$), so $C^*e^{\Cr{kappa:mane1}q_S}$ gets
% replaced by $\Cl{c:freede}e^{\Cr{kappa:mane}q_S}$, where
% $\Cr{c:freede}=\Cr{c:ce}(I^{**}) $. In this way, we get the estimate
% \begin{equation}\label{man1}
% (f_{a,b}^n)'(x)\geq \Cr{c:lfp}e^{\Cl[Kappa-const]{kappa:free} n},
% \end{equation}
% where $\Cr{kappa:free}=\min(\Cr{kappa:mane1}, \tilde{\kappa}/4,\Cr{kappa:mane})$. Note that $\Cr{c:lfp}$
% and $\Cr{kappa:mane}$ (and therefore $\Cr{kappa:outside}$) do not
% depend on $I$ and $\kappa_3=\kappa_2$.
% We now have (i) and (ii) with
% $\kappa_4=\min(\Cr{kappa:outside}',\Cr{kappa:free})$.

\begin{remark}\label{p-est-2} Note that we in this setting will have an
  analogy of Lemma \ref{le:ce} and the estimate 
  \begin{equation}\label{eq:p-est-2}
p\leq \frac {4r}{\kappa_4}
\end{equation}
holds.

\end{remark}
\medskip
% {\it Claim.} At the end of the last period $f_{a,b}^n(x)$ cannot be bound.

% {\it Proof.} Suppose by contradiction that $f_{a,b}^n(x)$ is still 
% bound to $f_{a,b}^{n-t_S}(c)$ then
% $$
% \text{\rm dist}(f_{a,b}^{n-t_S}(x),c) \geq \Cr{c:celbstretched} e^{-\alpha(n-t_S)}-e^{-\beta(n-t_S)}
% $$

% Since $f_{a,b}^{t_S}(x)\in I_{\pm r}$, $e^{-r}\geq \delta$ and 
% by Remark \ref{p-est-2},  $p\leq 4r/\Cr{kappa:outside}$. We conclude
% that
% $$
% \text{dist}(f_{a,b}^{n-t_S}(x),c)\geq
% \frac{\Cr{c:celbstretched}}{2}e^{-4r\alpha/\Cr{kappa:outside}}\geq
% \frac{\Cr{c:celbstretched}}{2}\delta^{4\alpha/\Cr{kappa:outside}}\gg \delta.
% $$
% It follows that $f_{a,b}^n(x)\not\in I$. The last bound period satisfies
% \eqref{pbounda0} and for the final free period we use 
% Lemma \ref{le:lowerboundret}. This completes
% the proof of the lemma.

% \end{proof}

\begin{remark}\label{rem:outsidedist}
We need in the future in several occasions a distorsion estimate in
the situation of Lemma \ref{mane}, i.e. for orbits located outside of
$I$. We need the estimate for parameter dynamics, i.e we have a
parameter interval $\omega$ in the space of $a$-parameters, and we
consider $\xi_j(\omega,b)$ for $j$ satisfying $\nu\leq j<\mu=n$, where
$\xi_j(\omega,b)\cap I=\emptyset$ for $j=\nu,\dots,n-1$ and
$\xi_n(\omega,b)\cap I\neq \emptyset$. Let $\omega'\subset \omega$ be the
interval that is mapped onto $I$. Then Lemma \ref{mane} (i) implies that
\begin{equation}\label{out:exp}
\inf_{a\in\omega'}(f_{a,b}^{n-\nu})'(f_{a,b}^\nu(c))
\geq \Cr{c:celb}e^{\Cr{kappa:outside}(n-\nu)}
\end{equation}
We also assume that \eqref{eq:lower} hold at time $\nu$ i.e. for $a\in\omega$
\begin{equation}\label{eq:stretched}
(f^j_{a,b})'(f_{a,b}(c))\geq \Cr{c:celbstretched}e^{j^{2/3}},\qquad 1\ \leq j\leq \nu-1
\end{equation}
Then by Corollary \ref{dera_cor}
\begin{equation}\label{deraceq1}
1\leq\frac{\partial_a\xi_\nu(a,b)}{(f_{a,b}^{\nu-1})'(f_{a,b}(c))}\le q_*.
\end{equation}

Then we conclude from Lemma \ref{le:growth} that 

\begin{equation}\label{eq:expgrowth}
|\xi_{n}(\omega',b)|\geq \frac1{q*}\inf_{a\in\omega'}
(f_{a,b}^{n-\nu})'(f_{a,b}^\nu(c))\cdot|\xi_\nu(\omega',b)|.
\end{equation}

\end{remark}

\begin{lemma}\label{le:outside}
There exists a constant $\Cl{c:freedist}$, such that  in the situation of
Remark~\ref{rem:outsidedist}, if $a',a''\in\omega'$ then
\begin{equation}\label{freed}
\frac{(f_{a',b}^{n-\nu})'(f^\nu_{a',b}(c))}{(f_{a'',b}^{n-\nu})'
(f^\nu_{a'',b}(c))}\leq\exp\left(\Cr{c:freedist}\frac{|f^{n}_{a',b}(c)
-f^{n}_{a'',b}(c)|}{\delta}\right).
\end{equation}
\end{lemma}

\begin{proof}
Set $x_k=f^{\nu+k}_{a',b}(c)$ and $y_k=f^{\nu+k}_{a'',b}(c)$. Note
that $f_{a,b}'(x)$ is independent of $a$. Therefore
\begin{align*}
\log\frac{(f_{a',b}^{n-\nu})'(x_0)}{(f_{a'',b}^{n-\nu})'(y_0)}
& = \sum_{k=0}^{n-\nu-1}(\log f'_{a',b}(x_k)- \log f'_{a',b}(y_k))\\
& \leq \sum_{k=0}^{n-\nu-1}\left(\frac{|f_{a',b}''(\eta_k)|}
{f_{a',b}'(\eta_k)}\cdot |x_k-y_k| \right)
\end{align*}
for some $\eta_k$ between $x_k$ and $y_k$.
Since $\eta_k\notin I^*$ for $k=0,\dots,n-\nu-1$, we get,
by~\eqref{secondb} and~\eqref{sec_der},
\begin{equation}\label{eq:nonlinearity}
\frac{|f_{a',b}''(\eta_k)|}{f_{a',b}'(\eta_k)}<
\frac{80}{\Cr{c:lder}\delta}.
\end{equation}
Therefore,
\begin{equation}\label{out1}
\log\frac{(f_{a',b}^{n-\nu})'(x_0)}{(f_{a'',b}^{n-\nu})'(y_0)}
\le\frac{80}{\Cr{c:lder}\delta}\sum_{k=0}^{n-\nu-1}|x_k-y_k|.
\end{equation}

We have $x_k=\xi_{\nu+k}(a',b)$ and $y_k=\xi_{\nu+k}(a'',b)$.
Therefore, by Remark~\ref{rem:outsidedist},
\[
|x_k-y_k|\le\frac{q_*}{\Cr{c:celb}} e^{\Cr{kappa:outside}(n-\nu-k)}
|x_{n-\nu}-y_{n-\nu}|.
\]
Thus,
\[
\sum_{k=0}^{n-\nu-1}|x_k-y_k|\le \frac{q_*}{\Cr{c:celb}} \sum_{m=0}^\infty
e^{-\Cr{kappa:outside} m}|x_{n-\nu}-y_{n-\nu}|=\frac{q_*}{\Cr{c:celb}(1-e^{-\Cr{kappa:outside}})}
|x_{n-\nu}-y_{n-\nu}|.
\]
Together with~\eqref{out1}, we get
\[
\log\frac{(f_{a',b}^{n-\nu})'(x_0)}{(f_{a'',b}^{n-\nu})'(y_0)}
\le\frac{80 q_*}{\Cr{c:lder} \Cr{c:celb}(1-e^{-\Cr{kappa:outside}})\delta}|x_{n-\nu}
-y_{n-\nu}|.
\]
and we have proved~\eqref{freed} with
\[
\Cr{c:freedist}=\frac{160 q_*}{\Cr{c:lder} \Cr{c:celb}(1-e^{-\Cr{kappa:outside}})}.
\]
\end{proof}

\begin {remark}\label{rem:out}
We note that the distortion in Lemma~\ref{le:outside} is uniformly
bounded since $|f^{n}_{a,b}(c)-f^{n}_{a',b}(c)|\leq 2\delta$.
\end{remark}

We will need a distorsion estimate
of the same type as Lemma \ref{le:outside} in the situation when we
only assume estimates as \eqref{out:exp} for all $\nu<n$  and
with another Lyapunov exponent $\Cl[Kappa-const]{kappa:free}>0$, together with
\eqref{eq:stretched}. This is the case of {\em hyperbolic times}
in the sense of Alves.

\begin{lemma}\label{le:hyptimedist}

 Assume that $\xi_j(\omega,b)$, $j=\nu,\dots,n$, is located
in $U=S^1\setminus I^{**}$ and 
\begin{equation}\label{out:exp1}
\inf_{a\in\omega}(f_{a,b}^{n-j})'(f_{a,b}^j(c))
\geq \Cr{c:celb}e^{\Cr{kappa:free}(n-j)}\quad \text{for all}\ j,\ \nu\leq j <n.
\end{equation}
Furthermore assume that \eqref{eq:stretched} is satisfied.

Then 

\begin{equation}\label{freed1}
\frac{(f_{a',b}^{n-\nu})'(f^\nu_{a',b}(c))}{(f_{a'',b}^{n-\nu})'
(f^\nu_{a'',b}(c))}\leq\exp\left(\Cl{c:alvesdist}|f^{n}_{a',b}(c)
-f^{n}_{a'',b}(c)|\right).
\end{equation}
Here the constant $\Cr{c:alvesdist}$ can be chosen as
$\Cr{c:alvesdist}={\Cr{c:alvesdist}'}N(f_{\mathcal N},U)/(1-e^{-\Cr{kappa:free}})$, where
$N(f_{\mathcal N},U)$ is the maximal nonlinearity

$$
N(f_{\mathcal N},U)=\sup_{(a,b)\in{\mathcal N}} \max_{x\in U} \frac{|f_{a,b}''(x)|}{f'_{a,b}(x)}.
$$

$N(f_{\mathcal N},U)$ depends only on
$f_{a_0}$ and hence not on $\delta$ and $\Cr{c:alvesdist}'$ is a constant that
only depends on $f_{a_0}$.
\end{lemma}
\begin{proof} We will not give the proof since it is virtually word by
  word the same as that of Lemma \ref{le:outside}. The only
  difference that the upper bound $80/(\Cr{c:lder}\delta)$   in
  \eqref{eq:nonlinearity} is replaced by $N({\mathcal N},U)$.

\end{proof}

\section{Induction}\label{sec:bound-free-ess}

Recall that the partition of the return interval
$I^*=(c-\delta,c+\delta)$ was introduced on Section~\ref{Partition}.

Recall also that we defined $\xi_n(a,b)=f^n_{a,b}(c)$.

The next lemma will be used for the startup of the induction.

\begin{lemma}\label{startupINT}
Assume that $\delta_1$ is sufficiently small and the neighborhood
$\caln$ of $(a_0,1)$ is sufficiently small. Then there are constants
$\Cr{c:cba-sqrt}$, $\Cr{c:celbstretched}$,
$\Cl[Kappa-const]{kappa:startup}>0$ so that for every
$\varepsilon=2^{-J_0}$ sufficiently small, there is a function
$b_0(\varepsilon)$ so that 
for every $b_0(\varepsilon) \leq b<1$ one can partition
$(a_0-\varepsilon,a_0-\varepsilon^2)$ into a partition ${\mathcal Q}$ of
countable number of parameter intervals $\omega$ and an exceptional
set ${\mathcal E}$ of measure $o(\epsilon)$, so that for all
$\omega\in {\mathcal Q}$ there is an $n_0=n_0(\omega)$
so that for some $(r,\ell)$, with $r\leq \sqrt{n_0}$, (or equivalently
$e^{-r}\geq e^{-\sqrt{n_0}}$)

$$
I_{r,\ell}\subset \xi_{n_0}(\omega,b)\subset I_{r,\ell}^+,
$$

and such that for every $a\in\omega$
\begin{itemize}
\item[(a)] $(f^j_{a,b})'(f_{a,b}(c))\geq \Cr{c:celbstretched} e^{\Cr{kappa:startup} j}$ for
$0\leq j \leq n_0-1$;
\item[(b)] $\partial_af^j_{a,b}(c)\geq \Cr{c:celbstretched} e^{\Cr{kappa:startup}(j-1)}$ for
$1\leq j \leq n_0$;
\item[(c)] $|\xi_j(a,b)-c|> \Cr{c:cba-sqrt}e^{-\sqrt{j}}$ for $1 \leq j < n_0$;
\item[(d)] $(f^{n_0-1}_{a,b})'(f_{a,b}(c))\geq e^{2(n_0-1)^{2/3}}$;
\item[(e)]   $|\xi_{n_0}(a,b)-c|\geq e^{-\sqrt{n_0}}$   
\end{itemize}

The corresponding statement holds also for the interval $(a_0+\varepsilon^2,a_0+\varepsilon)$.

\end{lemma}

\begin{proof} We partition $(a_0-\varepsilon,a_0-\varepsilon^2)$ into
  subintervals $\eta_j=(a_0-2^{-j},a_0-2^{-j-1})=(a_j',a_j'')$,
  $j=J_0,\dots,2J_0-1$. The critical point $c$ of unperturbed map
  $f_{a_0}$ is mapped to a repelling periodic point $P$ in $m$
  iterates. Let $U_0$ be a symmetric interval contained in the
  linearization domain of $P$ so that
  $$
  (f_{a,b}^\ell)'(x)\geq \lambda_1^\ell=\Lambda_1>1\qquad \text{for}\
  x\in U_0.
$$
Let $\tilde{\eta}_j=(a_0-2^{-j},a_0)$. Then there is a constant
$\Cl{c:perorb}$ so that 
$$(f_{a,b}^i)'(\xi_m(a,b))\geq \Cr{c:perorb}\lambda_1^i,\quad\text{for
  all}\ a\in\tilde{\eta}_j 
$$
as long as $\xi_{m+i}(\tilde{\eta}_j,b)\subset U_0$.
We now state a version of Lemma \ref{le:outside} which will be used in
the startup construction.

\begin{lemma}\label{le:outside1}

\end{lemma} Suppose that for $a\in\omega$ there is a constant
$\tilde{C}=\frac{1}{2}$, say, so that 

\begin{equation}\label{eq:expgrowth1}
|\xi_{m}(\omega',b)|\geq \frac1{q*}\inf_{a\in\omega'}
(f_{a,b}^{m-\nu})'(f_{a,b}^\nu(c))\cdot|\xi_\nu(\omega',b)|.
\end{equation}

% We can now use the standard hyperbolic distorsion estimate for
% expanding one-dimensional maps, see Lemma 1 of Shub-Sullivan
% \cite{ShSu}.

% \medskip
% {\bf Lemma} (\cite{ShSu}). Let $f$ be a $C^{1+\alpha}$ expanding endomorphism of $S^{1} .$ There is a constant $c>0$ such that if $I \subset S^{1}$ is an interval and $f^{n}$ is injective on I then
% $$
% \frac{1}{c}<\frac{\left|(f^{n})'(y)\right|}{\left|( f^{n})'(z)\right|}<c
% $$
% for any $y, z \in I$.

% \medskip
% Since the $x$- and $a$-derivatives are comparable within fixed
% constants according to Lemma \ref{le:growth} we  have  
and 
\begin{equation}\label{eq:exponential1}
(f^j_{a,b})'(f_{a,b}(c))\geq \tilde{C} e^{\frac{\hat{\kappa}}{4},},\qquad 1\ \leq j\leq \nu-1.
\end{equation}

Then there is a  constant $\Cr{c:freedist}$ so that 
\begin{equation}\label{freed2}
\frac{(f_{a',b}^{m-\nu})'(f^\nu_{a',b}(c))}{(f_{a'',b}^{m-\nu})'
(f^\nu_{a'',b}(c))}\leq\exp\left(\Cr{c:freedist}\frac{|f^{m}_{a',b}(c)
-f^{n}_{a'',b}(c)|}{\delta}\right).
\end{equation}

We will not give the  proof of this lemma since it is identical to
that of Lemma \ref{le:outside}. We conclude that

 $$
|\xi_{m+i}(\tilde{\eta}_j,b)|\geq \frac{\Cr{c:perorb}}{q_*}\lambda_1^i\cdot|\xi_m(\tilde{\eta}_j,b)|
$$
and $q_*$ has a uniform control by Corollary \ref{dera_cor}.

We also get a uniform distorsion control of
$\partial_af_{a,b}^\nu(\xi_j(a,b))$, i.e. there is a a constant $\tilde{C}$
depending only on $a_0$ so that for all $a$, $a'$ in $\tilde{\eta}_j$

\begin{equation}\label{initialdist}
\frac{1}{\tilde{C}}\leq
\frac{\partial_af_{a,b}^\nu(\xi_j(a,b))}{\partial_af_{a',b}^\nu(\xi_j(a',b))}\leq
\tilde{C},\qquad \nu=1,2,\dots,
\end{equation}
as long as $\xi_{m+\nu}(\tilde{\eta}_j,b)\subset U_0$.

It
follows that there is a first time $L$ so that

$\xi_{m+L+1}(\tilde{\eta}_j,b)\not\subset U_0$.
We write  $\tilde{\eta}_j$ as the disjoint union (except for an endpoint)
$$
\tilde{\eta}_j={\eta}_j\cup{\eta'}_j.
$$

By \eqref{initialdist} it follows that $\xi_{m+L }(\eta_j,b)$ and
$\xi_{m+L }(\eta_j',b)$ are comparable withing a  fixed constant 
$\Cl{c:radonnikodym}$, which only depends on $f_{a_0}$.

We continue to iterate $\xi_{m+L+i}(\eta_j,b)$ for $i=1,2,\dots$. By
Lemma \ref{mane}, Lemma \ref{le:growth} and the control of the constant
$q_*$ it follows that at the first time $J$ such that
$\xi_{m+L+J}(\eta_j,b)\cap I^*\neq\emptyset$

$$
|\xi_{m+L+J}(\tilde{\eta_j},b)|\geq \frac{\Cr{c:celb}}{q_*}e^{\Cr{kappa:outside}J}|\xi_{m+L}(\tilde{\eta_j},b)|.
$$

Then $\Cr{kappa:startup}=\min(\log \lambda_1,\Cr{kappa:outside})$ is the
required Lyapunov exponent in (a).
It follows  by Lemma \ref{mane} (ii), Lemma
\ref{le:growth} and the control of $q_*$ that the time $J$ will be
finite. At time $N_0=m+L+J$, we partition
$$
(c-\delta,c+\delta)\cap  \xi_{N_0}(\eta_j,b)
$$
into preimages $\{\omega\}$ under the map $a\mapsto \xi_{N_0}(a,b)$ of the
partition ${\mathcal Q}=\{I_{r,l}\}$ and define $n_0=N_0$ for these
$\omega$:s. In the special case when $\xi_{N_0}(\eta_j,b)$ only
intersects partially an end interval of ${\mathcal Q}=\{I_{r,l}\}$, we
just keep iterating until we cover complete intervals of ${\mathcal
  Q}$. In other special case when $\xi_{N_0}(\eta_j,b)$ only partially
covers a $I_{r,l}$ interval we adjoin the corresponding preimage to
the adjacent interval. Simultaneously we delete the part of $ \eta_j$ that is
mapped to $(c-e^{-\sqrt{n_0}},c+e^{-\sqrt{n_0}})$. By the uniform
distorsion of both the $x$-derivative and $a$-derivative which follows
from Lemma \ref{mane}, Lemma \ref{le:outside} and Corollary \ref{dera_cor} a
proportion of at most $\Cl{c:proportion}e^{-\sqrt{n_0}}/\delta$ of the piece of
$\eta_j$ mapped into $(c-e^{-\sqrt{n_0}},c+e^{-\sqrt{n_0}})$. Here $\Cr{c:proportion}$ is a 
constant only depending on $f_{a_0}$.  We continue to
iterate  $\xi_{N_0}(\eta_j,b)\setminus (c-\delta,c+\delta) $, still
using Lemma \ref{mane},  Lemma \ref{le:outside} and Corollary \ref{dera_cor}.
For the new returning interval  $\omega$ formed in this way
$n_0(\omega)>N_0$ and still only a quantity proportional to
$e^{-\sqrt{n_0}}/\delta$ is deleted.  The conclusions
(a)--(e) of Lemma \ref{startupINT} are immediately verified.
\end{proof}

\begin{remark} The startup argument is essentially the same as the
  free period argument in the main induction and the argument in Lemma
  \ref{mane} in Section \ref{Outside}. See the main induction
  below in this section for a more thorough discussion. The only difference is
  the initial period that is spent close to the repelling periodic
  point which in some sense replaces the bound period. The expansive
  behaviour close to the repelling periodic point allows us to avoid
  inessential free returns and gives the initial exclusion ratio of at most
  $\Cr{c:proportion}e^{-\sqrt{n_0}}/\delta$.

\end{remark}  

\medskip
  
Let us now fix $b$, $0< b_0(\varepsilon)\leq b <1$. Note that  for every positive integer $n$ we have
a family $\calp_n$ of subintervals of $(a_0-\eps, a_0+\eps)$ (as in
Lemma~\ref{startupINT}) with pairwise disjoint interiors, such that
each element of $\calp_{n+1}$ is contained in some element of
$\calp_n$. In the set of pairs $(n,\omega)$ such that
$\omega\in\calp_n$ there is a natural structure of a combinatorial
tree, that goes down with its branches. Pairs $(n,\omega)$ are
vertices of this tree; $n$ is the level on which the vertex lies;
there is an edge from $(n,\omega)$ to $(n+1,\omega')$ if and only if
$\omega'\subset\omega$.

Certain pairs with the property $\xi_n(\omega,b) \subset I^*$ will be
called \emph{free return pairs}. 

The induction will be separate on every branch of the tree. Fixing the
branch results in considering a descending sequence of intervals
$\omega_n\in\calp_n$. If $(n,\omega_n)$ is a  free return pair,
then we will call $n$ a free return time. An important feature
of the construction is that if $n$ is not a free return time then
$\omega_n=\omega_{n-1}$. The main induction step will be from a free return
time to the next free return time. The constants $\Cr{c:celbstretched}$ and $\Cr{c:cba-sqrt}$ are
as in Lemma~\ref{startupINT}. In the whole induction they will stay
the same.

\medskip
Our \textbf{Induction Statement} is the following. If $n$ is a free
return time and and $a\in\omega$, then:
\begin{itemize}
\item[(i)] we have
\begin{equation}\label{ih1}
(f_{a,b}^{n-1})'(f_{a,b}(c))\geq e^{2(n-1)^{2/3}},
\end{equation}
\item[(ii)] for every $\nu\in[n_0,n)$
\begin{equation}\label{ih2}
(f_{a,b}^\nu)'(f_{a,b}(c))\geq e^{\nu^{2/3}},
\end{equation}
\item[(iii)] for every $\nu\in[1,n)$
\begin{equation}\label{ih3}
(f_{a,b}^\nu)'(f_{a,b}(c))\geq \Cr{c:celbstretched} e^{ \nu^{2/3}},
\end{equation}
\item[(iv)] if $\nu<n$ is also a free return time, then
\begin{equation}\label{ih4}
(f^{n-\nu}_{a,b})'(f^{\nu}_{a,b}(c))\geq C(\delta)>> 1,
\end{equation}
\item[(v)] for every $\nu\in[n_0,n]$
\begin{equation}\label{BA}
|\xi_\nu(a,b)-c|\geq e^{- \sqrt{\nu}},
\end{equation}

\item[(vi)] for every $\nu\in[0,n]$
\begin{equation}\label{BA1}
|\xi_\nu(a,b)-c|\geq \Cr{c:cba-sqrt}e^{- \sqrt{\nu}},
\end{equation}
% \item[(vi)] if additionally $n$ is a good free return time, then
% \begin{equation}\label{ih6}
% (f_{a,b}^{n-1})'(f_{a,b}(c))\geq e^{\kappa_7(n-1)}.
% \end{equation}
\end{itemize}

In \cite{BC1} and \cite{BC2} statements (v) and (vi) is called the basic
assumption (BA).

Remember that $b$ sufficiently close to $1$ is fixed. We set
$\calp_n=\{\omega_b\}$ for $n=1,2,\dots,N_0$. Thus, this is the
beginning of every branch. Then we declare $n_0=n_0(\omega)$ to be the first free
return time according to the startup construction. Thus, for every
branch we have to start induction by checking that that the above
conditions are satisfied for $n=n_0(\omega)$.

\begin{lemma}\label{first_step} The Induction Statement
  conditions~(i)-(vi) are satisfied for $n=n_0$.
\end{lemma}

This is a consequence of the startup construction, Lemma \ref{startupINT}.

% \begin{proof}
% Since $\kappa_7>\kappa_6>\kappa_5>0$, \eqref{ih6} implies~\eqref{ih1}
% and~\eqref{ih2}. Condition~(iv) is empty in our case. Thus, we have to
% prove~\eqref{ih6}, (iii) and~(v).

% Condition~(iii) follows immediately from Lemma~\ref{startupINT}~(a)
% with any $\kappa_5\le\kappa_4$. However, we can take even
% $\kappa_7<\kappa_4$ and then~\eqref{ih6} holds provided $n_0$ is
% sufficiently large (which can be done without changing $C_{11}$ or
% $\kappa_4$). Inequality~\eqref{BA} holds for $\nu\in[1,n)$ by
% Lemma~\ref{startupINT}~(c). We get it for $\nu=n$ again by making
% $n_0$ sufficiently large.

% There is no problem with additional conditions, since $\kappa_5$ can
% be chosen arbitrarily small.
% \end{proof}

Now we make a small modification of Definition~\ref{BoundMT}.

\begin{definition}\label{defbound1}
Let $a'$ be the midpoint of the interval $\omega$ such that
\[
\xi_n(\omega,b)\subset I^+_{r,l}=I_{r,l-1}\cup I_{r,l}\cup
I_{r,l+1}
\]
for some $n,r,l$. We define the \emph{bound period} as the maximal
integer $p$ so that for all $j \leq p$, $a\in\omega$, and $x
\in\xi_n(\omega,b)$
\begin{equation}\label{BC}
|f_{a,b}^j(x)-f_{a',b}^j(c)|\leq e^{-4 \sqrt{j}}.
\end{equation}
\end{definition}

By~\eqref{derxbd} and Lemma~\ref{dera}, we get for every $n,a,b,x$
\begin{equation}\label{derabd}
\partial_a f_{a,b}^n(x)\le\sum_{k=0}^{n-1}4^k=\frac{4^n-1}3<4^n.
\end{equation}

%%
%%Let us assume that $\beta<2-\log 4$ and $\kappa_5<2$.
%%

In the several next lemmas we will be using the same set of
assumptions. We formalize these in the following definition
\begin{definition} We say that \emph{Condition (*)} is satisfied if
\begin{itemize}
\item $\omega$, $n,r,l,p$ and $a'$ are as in
Definition~\ref{defbound1},
\item conditions~(iii), (v) and (vi) of the Induction Statement hold.
\end{itemize}
\end{definition}

%We will call those assumptions \emph{Condition (*)}.

Next we formulate another version of the Bound Distorsion Lemma

\begin{lemma}\label{bounddist}
There is a constant $\Cl{c:bddist}$ such that if Condition (*) holds, then
\begin{equation}\label{estimate1}
\frac{1}{\Cr{c:bddist}} \leq \frac{(f_{a,b}^k)'(f_{a,b}(y))}
{(f_{a,b}^k)'(f_{a,b}(c))} \leq \Cr{c:bddist}
\end{equation}
and
\begin{equation}\label{estimate2}
\frac{1}{\Cr{c:bddist}} \leq \frac{(f_{a,b}^k)'(f_{a,b}(y))}
{(f_{a',b}^k)'(f_{a',b}(c))} \leq \Cr{c:bddist}
\end{equation}
for every $x\in I_{r,l}$, $y$ between $x$ and $c$, $a\in\omega$, and
$k\le \max(p,n)$.
By making $\delta$ sufficiently small, the constant $\Cr{c:bddist}=\Cr{c:bddist}(\delta)>1$  can be chosen arbitrarily close
to $1$.

\end{lemma}

\begin{proof} The proof will proceed by induction on $k$.
Using~\eqref{derxbd}, we get in the same way as in the proof of
Lemma~\ref{bounddistMT}
\begin{equation}\label{bd0}
\frac{(f_{a,b}^k)'(f_{a,b}(y))}{(f_{a,b}^k)'(f_{a,b}(c))} \leq
\exp \left( \sum_{j=1}^{k} \dfrac{13 |f_{a,b}^j(y)-f_{a,b}^j(c)|}
{f_{a,b}'(f_{a,b}^j(c))} \right).
\end{equation}

Furthermore,
\begin{equation}\label{bound-term} 
|f_{a,b}^j(y)-f_{a,b}^j(c)| \le |f_{a,b}^j(x)-f_{a,b}^j(c)| \leq
|f_{a,b}^j(x)-f_{a',b}^j(c)| + |f_{a',b}^j(c)-f_{a,b}^j(c)|,
\end{equation}
and by
\eqref{BC} we have 

\begin{equation}
|f_{a,b}^j(x)-f_{a',b}^j(c)| \leq e^{-4 \sqrt{j}}.
\end{equation}
Thus, we need to estimate $|f_{a',b}^j(c)-f_{a,b}^j(c)|$. 
Note that by the mean value theorem there is $a''$ between $a$ and
$a'$ so that 

\begin{equation}\label{mean-value}
|f_{a',b}^j(c)-f_{a,b}^j(c)|=\partial_a f^j_{a'',b}(c)\cdot |a-a'|.
\end{equation}

Note that  $|a-a'|$ can be interpreted as $|\xi_1(a,b)-\xi_1(a',b)|$.
By Lemma \ref{le:growth}
$$
|\xi_n(a,b)-\xi_n(a',b)|\geq
\frac{1}{q_*}\inf_{\tilde{a}\in[a,a']}(f_{\tilde{a},b}^{n-1})'(f_{\tilde{a},b}(c))\cdot
|\xi_1(a,b)-\xi_1(a',b)|. 
$$
By the induction statement (iii),
$(f_{\tilde{a},b}^{n-1})'(f_{\tilde{a},b}(c))\geq \Cr{c:celbstretched} e^{{(n-1)^{2/3}}}$.
We may therefore conclude that 
$|a-a'|\leq\Cr{c:celbstretched}^{-1}  q_*e^{-(n-1)^{2/3}}\cdot |\xi_n(a,b)-\xi_n(a',b)|$.
But by the mean value theorem

\begin{equation}\label{eq:diffj}
|f_{a,b}^j(x)-f_{a',b}^j(c)|=|f_{a,b}^{j-1}(f_{a,b}(x))-f_{a,b}^{j-1}(f_{a,b}(c))|
 =(f^{j-1}_{a,b})'(f_{a,b}(y))\cdot |f_{a,b}(c)-f_{a,b}(x)|.
\end{equation}

However since $|\xi_n(a,b)-\xi_n(a',b)|\leq e^{-r}$, we have
\begin{equation}\label{eq:diff1}
|f_{a,b}(c)-f_{a,b}(x)|\geq \Cr{c:lcr} \cdot |x-c|^3\geq \Cr{c:lcr}
e^{2(-r-1)}\cdot e^{-1}\cdot|\xi_n(a,b)-\xi_n(a',b)|.
\end{equation} 
By the basic assumption $e^{-r}\geq \Cr{c:cba-sqrt}e^{-\sqrt{n}}$. 
Note also that by Corollary \ref{dera_cor},
$\partial_af^j_{a'',b}(c)$ is comparable within the multiplicative constant $q_*$ to
$(f^{j-1}_{a'',b})'(f_{a'',b}(c))$. But this quantity is in itself by
induction comparable within a multiplicative constant $\Cr{c:bddist}$ to
$\inf_{\tilde{a}\in[a,a']}(f_{\tilde{a},b}^{n-1})'(f_{\tilde{a},b}(c))$. We 
use the statement of our result for $k=j-1$. 
% Note also that by the mean value theorem there is $a''$ so that 
% \begin{equation}
% |f^j_{a',b}(c)-f^j_{a,b}(c)|=\partial_a f^j_{a''}(c)|a-a'|
% \end{equation}

We use \eqref{mean-value} and
note that $|a-a'|$ also can be written as $|\xi(a,b)-\xi_1(a',b)|$.
By Lemma \ref{le:growth}
\begin{equation}\label{eq:timen}
|\xi_n(a,b)-\xi_n(a'b)|\geq \frac{1}{q_*}\inf_{\tilde{a}\in[a,a']}(f^{n-1})'(f_{\tilde{a},b}(c))
\end{equation}

By combining \eqref{eq:diffj},\eqref{eq:diff1} and \eqref{eq:timen}, we obtain
\begin{align*}
|f^j_{a,b}(x)-f^j_{a,b}(c)|&=|f_{a,b}^{j-1}(f_{a,b}(x))-f_{a,b}^{j-1}(f_{a,b}(c))|\geq(f^{j-1})'(f_{a,b}(y))|f_{a,b}(x)-f_{a,b}(c)| 
  \\
                           &\geq\inf_{y\in I_r}(f^{j-1}_{a,b})'(f_{a,b}(y))\,|f_{a,b}(x)-f_{a,b}(c)|\\
                           &\geq C_7 e^{-3r}\inf_{y\in I_r}(f^{j-1}_{a,b})'(f_{a,b}(y))\\
  &\geq C_7
    e^{-2r}\inf_{y\in I_r}(f^{j-1}_{a,b})'(f_{a,b}(y))|\xi_n(a,b)-\xi_n(a',b)|\\
  &\geq C_7e^{-2r}
    \inf_{y\in I_r}(f^{j-1}_{a,b})'(f_{a,b}(y))(\inf_{\tilde{a}\in[a,a']}\partial_af^{n}_{a'',b}(c))|\xi_1(a,b)-\xi_1(a',b)|\\
  &\geq C_7 e^{-2r} \frac{ \inf_{y\in I_r}(f^{j-1}_{a,b})'(f_{a,b}(y))}{ 
    \sup_{z\in I_r}(f^{j-1}_{a,b})'(f_{a,b}(z))}
  (\inf_{\tilde{a}\in[a,a']}\partial_af^{n}_{a'',b}(c))|)\,|\xi_j(a,b)-\xi_j(a',b)|.
\end{align*}
Now $\sup_{z\in I_r}(f^{j-1}_{a,b})'(f_{a,b}(z))$ and $\inf_{y\in
  I_r}(f^{j-1}_{a,b})'(f_{a,b}(y))$ are comparable with constant
$\Cr{c:bddist}^2$, by the statement of Lemma \ref{bounddist} with
$k=j-1$. This is where the inductive step is used. Moreover 
by Corollary \ref{dera_cor} and the induction statement (i), we have
$\partial_af_{a,b}^n(c)\geq q_*^{-1} e^{2(n-1)^{2/3}}$. Furthermore $e^{-2r}\geq
e^{-2\sqrt{n}}$, by the induction statement (v), \eqref{BA}. Combining these estimates
we get
$$
|\xi_j(a,b)-\xi_j(a',b)|\leq \frac{1}{2}|f_{a,b}^j(x)-f_{a,b}^j(c)|
$$
since $n\geq n_0(a)$ and $n_0(a)$ can be choosen arbitrarily large.

  %%  \inf_{\tilde{a}\in[a,a']}\partial_af^n_{a,b}(c).

%

%\begin{align*}
% |f_{a',b}^j(c)-f_{a,b}^j(c)|&\leq
%                               \Cr{c:cba-sqrt}^{-2}\Cr{c:celbstretched}^{-1}\Cr{c:bddist}^2q_*^2\cdot
%                               e \cdot e^{2\sqrt{n}}e^{-n^{2/3}}|f_{a,b}(x)-f_{a,b}(c)|\\
% &\leq \frac{1}{2}|f_{a,b}(x)-f_{a,b}(c)|,
% \end{align*}
% Since $\alpha=\kappa_5/400$  satisfies $\alpha \leq
% \frac{\kappa_5}4$ and we know that $n\geq n_0$, with $n_0=N_0$  is arbitrarily large,
% as $|a-a_0|\to 0$, $a\in\omega$.
%since $n\geq n_0(a)$, and $n_0(a)$ can be chosen arbitrarily large.

When inserting this estimate in \eqref{bound-term} we conclude that 
\begin{equation}\label{bd1}  
|f_{a,b}^j(y)-f_{a,b}^j(c)|\leq 2e^{-4 \sqrt{j}}.
\end{equation}

To estimate $f_{a,b}'(f_{a,b}^j(c))$ from below, we use~(vi) of the
Induction Statement and~\eqref{secondb}. We get
\begin{equation}\label{bd2}
f_{a,b}'(f_{a,b}^j(c))\ge \Cr{c:lder} \Cr{c:cba-sqrt}^2 e^{-2\sqrt{j}}.
\end{equation}
Putting together~\eqref{bd0},~\eqref{bd1} and~\eqref{bd2}, we obtain

\begin{equation}\label{bd3}
\frac{(f_{a,b}^k)'(f_{a,b}(y))}{(f_{a,b}^k)'(f_{a,b}(c))}<
\exp\left(\frac{13\cdot 2}{C_9\cdot C_1^2}\sum_{j=1}^k
\frac{e^{-4\sqrt{j}}}{e^{-2\sqrt {j}}}\right).
\end{equation}

For the lower bound, we obtain in a similar way as~\eqref{bd0}
\[
\frac{(f_{a,b}^k)'(f_{a,b}(c))}{(f_{a,b}^k)'(f_{a,b}(y))} \leq
\exp \left( \sum_{j=1}^{k} \dfrac{13|f_{a,b}^j(y)-f_{a,b}^j(c)|}
{(f_{a,b})'(f_{a,b}^j(y))} \right).
\]
Note however that
\[
(f_{a,b})'(f_{a,b}^j(y)) \geq \Cr{c:lder}(f_{a,b}^j(y) - c)^2 \geq
\Cr{c:lder} \left( |c-f_{a,b}^j(c)| - |f_{a,b}^j(c) - f_{a,b}^j(y)|
\right)^2,
\]
and using~\eqref{BA} and~\eqref{bd1} we get
\[
f_{a,b}'(f_{a,b}^j(y))>\Cr{c:lder} (\Cr{c:cba-sqrt} e^{-\sqrt{j}}-2e^{-4\sqrt{ j}})^2
\]
whenever $\Cr{c:cba-sqrt} e^{-\sqrt {j}}>2e^{-4\sqrt{ j}}$. Now, $\Cr{c:cba-sqrt}$
is fixed and thus, there is a
positive integer $\tilde N$ such that if $j\ge\tilde N$ then
$\Cr{c:cba-sqrt}e^{-\sqrt{ j}}>2\times 2e^{-4\sqrt {j}}$, and then
\[
f_{a,b}'(f_{a,b}^j(y)) > \frac{\Cr{c:lder} \Cr{c:cba-sqrt}^2}4 e^{-2\sqrt{ j}}.
\]
By making $\caln$ and $I^*$ sufficiently small, we can make
$|f_{a,b}^j(c) - f_{a,b}^j(y)|$ smaller than $\Cr{c:cba-sqrt}e^{-4 \sqrt{j}}$
instead of $2e^{-4\sqrt {j}}$, and then we get
\[
f_{a,b}'(f_{a,b}^j(x)) \geq \Cr{c:lder} \Cr{c:cba-sqrt}^2\Cl[K-const]{k:K10} e^{-2\sqrt{j}}
\]
for some constant $\Cr{k:K10}$ depending only of $a_0$. In such a
way, in the same way as we obtained~\eqref{bd3}, we get a similar
estimate for the reciprocal ratio, but with a different constant. We
choose then as $\Cr{c:bddist}$ the larger of those constants and we
get~\eqref{estimate1}. The proof of \eqref{estimate2} is completely
analogous and will be omitted. As for the statement that $\Cr{c:bddist}$ can
be chosen arbitrarily close to $1$ (but larger than $1$), we refer
to the argument in Lemma \ref{bounddistMT}. 
\end{proof}

\begin{lemma}\label{LengthBP}
Assume that Condition (*) holds. Then 
the bound period $p$ in the sense of Definition \ref{defbound1} satisfies 

\[
p\leq 8r^{3/2}.
\]
\end{lemma}

\begin{proof}
We claim that $p\leq 8r^{3/2}$. Note that $\leq 8n^{3/4}<n$. We argue
by contradiction. Assume that there is $k>8r^{3/2}$ so that
$f^j_{a,b}(f_{a.b}(z)$ is still bound to $f^j_{a,b}(f_{a.b}(c)$
for all  $x\in I_{r,l}$ and all $z$ between $x$ and $c$.. By the Mean Value Theorem, there is a point $y$ between $x$
and $c$ such that
\[
|f_{a,b}^k(x)-f_{a,b}^k(c)|=|f_{a,b}(x)-f_{a,b}(c)|\cdot (f_{a,b}^{k-1})'(f_{a,b}(y)).
\]
By Lemma~\ref{bounddist} and~(iii) of the Induction Statement,
\[
(f_{a,b}^{k-1})'(f_{a,b}(y))\ge\frac1{\Cr{c:bddist}}(f_{a,b}^{k-1})'(f_{a,b}(c))\ge
\frac{\Cr{c:celbstretched}}{\Cr{c:bddist}}e^{(k-1)^{2/3}}.
\]
Since $|x-c|\ge e^{-r-1}$, by~\eqref{secondb} we get
\[
|f_{a,b}(x)-f_{a,b}(c)|\ge\frac{\Cr{c:lder}}3 e^{-3r-3}.
\]
Putting the last three inequalities together, we get
\[
|f_{a,b}^k(x)-f_{a,b}^k(c)|\ge \frac{\Cr{c:celbstretched}}{\Cr{c:bddist}}e^{{(k-1)}^{2/3}}
\cdot\frac{\Cr{c:lder}}3 e^{-3r-3}.
\]
Taking into account~\eqref{bd1} (which is valid also for $y=x$), we
get
\[
2>2e^{-\sqrt{k}}>\frac{\Cr{c:celbstretched}\Cr{c:lder}}{3\Cr{c:bddist}}e^{(k-1)^{2/3}}
e^{-3r-3}.
\]
Therefore,
\[
 (k-1)^{2/3}<3r-\log \Cl[K-const]{k:K11}
\]
where $\Cr{k:K11}$ is a constant only depending on $a_0$.

If $\delta$ is sufficiently small, then $\frac{1}{2}r>-\log \Cr{k:K11}$, and we get
$ k^{2/3}<(k-1)^{2/3}+\frac{2}{3}k^{-1/3}<4r$.We conclude that $k\leq
8r^{3/2}$ and this gives a contradiction.

% By~(iv) of the Induction Statement and the definition of $\omega$, we
% have $e^{-r-1}\ge C_{12} e^{-\alpha n}$, so
% \[
% r+1\le\alpha n-\log C_{12}.
% \]
% Since $\alpha=\kappa_5^2/400$, we have $\kappa_5^2/2-4\alpha>
% \kappa_5^2/3$, so
% \[
% k<\frac{4r}{\kappa_5}\le\frac{4\alpha n-4(1+\log C_{12})}{\kappa_5}
% <\frac{\kappa_5 n}2,
% \]
% provided
% \[
% n>\frac{-12(1+\log C_{12})}{\kappa_5^2}.
% \]
% However, this condition is satisfied if $\caln$ and $I^*$ are
% sufficiently small.
\end{proof}

% Under the assumptions of Lemma~\ref{LengthBP} we get that for every
% $k\le\min(p,\kappa_5 n/2)$ we have $k<\kappa_5 n/2$. This means that
% $p<\kappa_5 n/2$, and in particular, $\min(p,\kappa_5 n/2)=p$. Thus,
% from Lemmas~\ref{bounddist} and~\ref{LengthBP} we obtain the following
% corollary.

% \begin{corollary}\label{bdlen}
% Assume that Condition (*) holds. Then
% \begin{enumerate}
% \item inequality~\eqref{estimate1} holds for every $x\in I_{r,l}$, $y$
% between $x$ and $c$, $a\in\omega$, and $k\le p$;
% \item we have
% \[
% p\le\frac4{\kappa_5}r \quad\text{and}\quad p<\frac{\kappa_5n}2.
% \]
% \end{enumerate}
% \end{corollary}

Let us prove an elementary lemma about our family.

\begin{lemma}\label{elem}
For the family of double standard maps, if $0<|x-c|<1/2$ then
\begin{equation}\label{eq-elem}
f_{a,b}'(x)>\frac{|f_{a,b}(x)-f_{a,b}(c)|}{|x-c|}.
\end{equation}
\end{lemma}

\begin{proof}
We have $c=1/2$ and $f_{a,b}''(t)=-4\pi b\sin(2\pi t)$. Therefore
$f_{a,b}$ is strictly convex in $(c,c+1/2)$, and thus for
$x\in(c,c+1/2)$ we get~\eqref{eq-elem}. Similarly, in $(c-1/2,c)$ the
function $f_{a,b}$ is strictly concave, and~\eqref{eq-elem} follows
similarly.
\end{proof}

\begin{lemma}\label{deraftbd}
There exists a positive constant $\Cl{c:lowerpdir}$ such that if Condition (*)
holds, then
\begin{equation}\label{aaa0}
(f_{a,b}^{p+1})'(x)>\Cr{c:lowerpdir}e^r\cdot e^{-4 \sqrt{p}}.
\end{equation}
\end{lemma}

\begin{proof}
By Definition~\ref{defbound1}, there exists $a\in\omega$ such that
\begin{equation}\label{aaa1}
|f_{a,b}^{p+1}(x)-f_{a',b}^{p+1}(c)|\geq e^{-4 \sqrt{p+1}}.
\end{equation}

We have
\begin{equation}\label{aaa2}
  |f_{a,b}^{p+1}(x)-f_{a',b}^{p+1}(c)| \leq
|f_{a,b}^{p+1}(x)-f_{a,b}^{p+1}(c)| + |f_{a,b}^{p+1}(c)-f_{a',b}^{p+1}(c)|.
\end{equation}
By the Mean Value Theorem, there is a point $y$ between $x$
and $c$ such that
\begin{equation}\label{aaa3}
|f_{a,b}^{p+1}(x)-f_{a,b}^{p+1}(c)| =|f_{a,b}(x)-f_{a,b}(c)| \cdot
(f_{a,b}^p)'(f_{a,b}(y)).
\end{equation}

Now we estimate the second summand in~\eqref{aaa2}.
As in the proof of Lemma~\ref{bounddist}, we can prove that 

$$
|f_{a,b}^{p+1}(c)-f_{a',b}^{p+1}(c)|\leq\frac12|f^{p+1}_{a',b}(c) -f^{p+1}_{a,b}(x)|. 
$$
Therefore
$$
|f_{a,b}^{p+1}(x)-f_{a,b}^{p+1}(c)|\geq
\frac12|f_{a,b}^{p+1}(x)-f_{a',b}^{p+1}(c)|
\geq \frac12e^{-4\sqrt{p+1}}.
$$

% \[
% |f_{a,b}^{p+1}(c)-f_{a',b}^{p+1}(c)|\le 4^{p+1}|a-a'|<K_9 4^{p+1}
% e^{-\kappa_5 n}.
% \]
% By Corollary~\ref{bdlen},
% \[
% K_9 4^{p+1}e^{-\kappa_5 n}\le 4K_9 4^{\kappa_5 n/2}
% e^{-\kappa_5 n}=4K_9 e^{(\log 2-1)\kappa_5 n}<4K_9 e^{-(3/10) \kappa_5 n},
% \]
% so
% \[
% |f_{a,b}^{p+1}(c)-f_{a',b}^{p+1}(c)|< 4K_9 e^{-(3/10) \kappa_5 n}.
% \]
% Together with~\eqref{aaa1},~\eqref{aaa2} and~\eqref{aaa3}, we get
% \[
% e^{-\beta (p+1)}\le |f_{a,b}(x)-f_{a,b}(c)| \cdot
% (f_{a,b}^{p})'(f_{a,b}(y))+4K_9 e^{-(3/10) \kappa_5 n}.
% \]

% Again by Corollary~\ref{bdlen}, $\beta p\le (\beta/2)\kappa_5 n$.
% Recall that we assumed (just before Lemma~\ref{bounddist}) that
% $\kappa_5<2$. Therefore $\beta=\kappa_5^2/100<1/25$, so $\beta
% p<(1/50)\kappa_5 n$, and thus
% \[
% e^{-\beta (p+1)}>e^{-\beta}\cdot e^{-(1/50)\kappa_5 n}.
% \]
% Since we may assume that $n$ is sufficiently large, we get
% \[
% e^{-(1/50)\kappa n}>e^\beta\cdot 8K_9 e^{-(3/10) \kappa_5 n},
% \]
% and therefore
% \[
% |f_{a,b}(x)-f_{a,b}(c)| \cdot (f_{a,b}^{p})'(f_{a,b}(y))>
% e^{-\beta (p+1)}-4K_9 e^{-(3/10) \kappa_5 n}>
% \frac{e^{-\beta}}2 e^{-\beta p}.
% \]

{}From Lemma~\ref{bounddist} we get
\[
(f_{a,b}^{p})'(f_{a,b}(x))\ge\frac1{\Cr{c:bddist}^2}
(f_{a,b}^{p})'(f_{a,b}(y)),
\]
so
\[
(f_{a,b}^{p})'(f_{a,b}(x))>\frac{1}{2\Cr{c:bddist}^2}\cdot
e^{-4\sqrt{ p+1}} \cdot \frac1{|f_{a,b}(x)-f_{a,b}(c)|}.
\]
By the Chain Rule and Lemma~\ref{elem} we get from this inequality
\[
(f_{a,b}^{p+1})'(x)>\frac{1}{2\Cr{c:bddist}^2}\cdot
e^{-4\sqrt{ p+1}} \cdot \frac1{|x-c|}.
\]
Since $x\in
I_r$, we have $|x-c|\le e^{-r}$, and we get~\eqref{aaa0} with a
suitable choice of $\Cr{c:lowerpdir}$.
\end{proof}

Let $(n,\omega)$ be a free return pair. Consider the intervals
$\xi_{n+p+1+s}(\omega,b)$, $s=0,\ldots,s_0-1$, where $s_0$ is the
smallest nonnegative integer such that
\[
\xi_{n+p+1+s_0}(\omega,b) \cap I^* \neq \varnothing.
\]
For $0 \leq s < s_0,$ we say that $\xi_{n+p+1+s} (\omega,b)$ is in
free orbit and the length of this orbit is $s_0$. We also use the
notation $n'=n+p+1+s_0$ and it is our new free return time.

At the first free return there are different cases that can occur.

\medskip {\it Case 1.} $\Omega_{n'}=\xi_{n'}(\omega,b)$ is completely
contained in $I^*$ but does not contain a complete interval
$I_{r,\ell}$. Then either $\xi_{n+p+1+s_0}(\omega,b)$ is contained in
an interval $I_{r,l}$ or it is contained in the union of two adjacent
intervals $I_{r,l} \cup I_{r,l+1}$.

This is called an {\it inessential free return}. In this case
$\omega\in{\mathcal P}_{n'}$, and we just continue to iterate. This
also applies if $\Omega_{n'}$ intersects the boundary of $I^*$ but
does not contain any of the end intervals.

\medskip {\it Case 2.} $\Omega_{n'}$ contains at least one of the
partition intervals $I_{r,\ell}$. This is the case of an {\it
  essential free return}. We then proceed to define a new partition on
a subset of $\omega$ according to the following algorithm.
\begin{itemize}
\item We do not include the preimage of $(c-
e^{-\sqrt{ n'}},c+e^{-\sqrt{ n'}})$ under $a\mapsto\xi_{n'}(a,b)$
in $\bigcup_{\omega'\in\calp_{n'}}\omega'$, in order that~(BA) should
be satisfied.

\item The intervals $\omega_{r,\ell}$ and $\omega_{r,\ell}'$ are
defined as the preimages of $I_{r,\ell}$ under $\omega\ni
a\mapsto\xi_{n'}(a,b)$. Because of the double covering property of
$f_{a,b}$, there could be 0,1 or 2 such intervals. These will be new
partition intervals of ${\mathcal P}_{n'}$. At the two ends of
$\omega$ we could have the property that some intervals only partially
cover $I_{r,\ell}$. In that case we use the special rule that we
adjoin the corresponding subintervals to the adjacent intervals of
${\mathcal P}_{n'}$.

\item There may be at most three subintervals of $\omega$, call them
$\omega_1$, $\omega_2$ and $\omega_3$ which are mapped outside $I^*$
by $\omega\ni a \mapsto\xi_{n'}(a,b)$. In the beginning of the
procedure there are at most two intervals mapped outside, but in later
stages because of the double covering property of $f_{a,b}$, there can
be three. In this case these intervals are {\it long}, i.e they are
not contained in intervals adjacent to the end intervals in the
partition of $(c-e^{-r_\delta+1},c+e^{-r_\delta+1})$, they are considered to be
still free, and the free period continues for these intervals. If one
or more of the intervals $\omega_1$, $\omega_2$ or $\omega_3$ are {\it
  short}, i.e not long, they are adjoined to their adjacent neighbor.
\end{itemize}

Let $X_{\text{BA}}$ be the set that is mapped to $(c-e^{-\sqrt{n'}},c+e^{-\sqrt{n'}})$. Then we define the partition
${\mathcal P}_{n'}|(\omega\setminus X_\text{BA})$ as the intervals
$\{\omega_{r,\ell}\}$, $\{\omega_{r,\ell}'\}$ and $\omega_i$,
$i=1,2,3$. Some of these intervals may be empty.

Later we will see that deletions because of (BA) do not happen in
Case~1, because the interval $\Omega_{n'}$ is too long.

In order to proceed, we need to verify, at least partially, the
induction step from time $n$ to time $n'$. Here $n'$ is interpreted as
the first free return to $I^*$ after $n$. There may be previous
returns $\nu$, where another partition element of ${\mathcal P}_{\nu}$
has a free return, while the present parameter interval does not
return.

\begin{lemma}\label{le:stepind}
Assume the Induction Statement (i)-(vi). Then the Induction Statement
conditions~(i),(ii), (iii) and~(iv) hold for any free return pair
$(n',\omega')$, where $n'$ is as above. 
\end{lemma}

\begin{proof}
Let $\eta$ be the distance from $\xi_n(\omega,b)$ to $c$. Therefore,
by Induction Statement~(i) and~\eqref{secondb},
\[
(f_{a,b}^n)'(f_{a,b}(c))>\Cr{c:lder}\eta^2 e^{2(n-1)^{2/3}}.
\]
However, by~(v), $\eta\ge \Cr{c:cba-sqrt}e^{- \sqrt{n}}$, so we get
\begin{equation}\label{si1}
(f_{a,b}^n)'(f_{a,b}(c))>\Cr{c:lder} \Cr{c:cba-sqrt}^2e^{2(n-1)^{2/3}}e^{-2\sqrt{n}}.
\end{equation}

After time $n$ there follows the bound period $p$, and by
Lemma~\ref{bounddist} and (iii) we get
\begin{equation}\label{si2}
(f^k_{a,b})'(f^{n+1}_{a,b}(c))\ge \Cr{c:bddist}^{-1}(f^k_{a,b})'(f_{a,b}(c))\ge
\Cr{c:bddist}^{-1}\Cr{c:celbstretched}e^{ k^{2/3}}
\end{equation}
for all $k\le p$. Combining~\eqref{si1} and~\eqref{si2}, we conclude
that
\begin{equation}\label{si3}
(f^{n+k}_{a,b})'(f_{a,b}(c))=(f^n_{a,b})'(f_{a,b}(c))\cdot
(f^k_{a,b})'(f^{n+1}_{a,b}(c))>\Cr{c:lder} \Cr{c:cba-sqrt}^2\Cr{c:bddist}^{-1}\Cr{c:celbstretched}
e^{2(n-1)^{2/3}-4\sqrt{n}+ k^{2/3}}.
\end{equation}
For $k\leq p \leq 8r^{3/2}\leq 8 n^{3/4}$, we conclude that

$$
(f^{n+k}_{a,b})'(f_{a,b}(c))\geq e^{(n+k)^{2/3}}
$$

At time $n+p+1$ the bound period has expired and we have for all $a\in\omega$

\begin{equation}\label{est:afterbound}
(f^{p}_{a,b})'(f_{a,b}^{n}(c))e^{-3r}\geq \Cl{c:afterbound}e^{-4\sqrt{p+1}},
\end{equation}
where $\Cr{c:afterbound}$ is a constant only depending on $f_{a_0}$.

For the total derivative we obtain
\begin{equation}
(f^{n+p})'( f_{a,b}^{n}(c))   \geq \Cr{c:afterbound} e^{2n^{2/3}}e^{-2r}(f^{p}_{a,b})'(f_{a,b}^{n-1}(c))
\end{equation}
After raising \eqref{est:afterbound} to the power $\frac{2}{3}$
we obtain
\begin{equation}
(f^{n+p})'(f^{n}_{a,b}(c))\geq \Cr{c:afterbound}^{2/3}  e^{2n^{2/3}}(f^{p}_{a,b})'(f_{a,b}^{n}(c))^{1/3}e^{-\frac{8}{3}\sqrt{p+1}}
\end{equation}
Looking at the exponents, we get using \eqref{ih3} the lower bound
$$
2n^{2/3}+ \frac{1}{3}p^{2/3}-\frac{8}{3}\sqrt{p+1}\geq 2(n+p)^{2/3}+\frac{1}{10}p^{2/3}
$$
Here we have used the information from Lemma \ref{LengthBP}, $p\leq
8n^{3/4}$ 
and that if $p\leq\frac{1}{100}n$
$$
2n^{2/3}+\frac{1}{3}p^{2/3}\geq 2(n+p)^{2/3}+\frac{1}{5}p^{2/3}.
$$

Now, if $k=p+s$ and $0<s\le s_0$, then we can use
Lemma~\ref{mane}~(ii). If $s\ge M$ then
\[
(f_{a,b}^s)'(f_{a,b}^{n+p+1}(c))\ge e^{\Cr{kappa:outside} s}.
\]
If $s<M$ we use Lemma \ref{le:lowerboundret}, which allows a perturbation
to $(a,b)\in {\mathcal N}$ with a worse constant $\overline{d}/4$  instead
of   $\overline{d}/2$ and we get

\[
(f_{a,b}^s)'(f_{a,b}^{n+p+1}(c))\ge \frac{\overline{d}}{4}.
 \]

% If $s<M$, then using the fact that $I^*$ is a neighborhood of $c$, we
% get a constant $C_{16}\in(0,1]$ such that if $x\notin I^*$ then
% $f_{a,b}'(x)\ge C_{16}$. In such a way we get
% \[
% (f_{a,b}^s)'(f_{a,b}^{n+p+1}(c))\ge C_{16}^M.
% \]
Thus, independently whether $s\ge M$ or $s<M$, we have
\[
(f_{a,b}^s)'(f_{a,b}^{n+p+1}(c))\ge \frac{\overline{d}}{4} e^{(s-M)\Cr{kappa:outside}}.
\]
Together with~\eqref{si3} (where we substitute $k=p+s$), we get
\begin{align*}
(f^{n+k}_{a,b})'(f_{a,b}(c))&=(f^{n+p}_{a,b})'(f_{a,b}(c))\cdot
(f_{a,b}^s)'(f_{a,b}^{n+p+1}(c))\\
&>\Cr{c:lder} \Cr{c:cba-sqrt}^2\Cr{c:bddist}^{-1}\Cr{c:celbstretched}\frac{\overline{d}}{4}
e^{2(n-1)^{2/3}-4\sqrt{n} +\frac{1}{10} p^{2/3}+(s-M)\Cr{kappa:outside}}.
\end{align*}
%%Note that the constand $C_{16}$ depends on $\delta$. 

%However $n\geq
%n_0$ which can be taken arbritrarily large (depending on $\epsilon$)
%without changing $\delta$.
Note that since the constants $\Cr{c:ce}$, $\Cr{c:cba-sqrt}$, $\Cr{c:bddist}$, $\Cr{c:celbstretched}$ are
absolute constants and $p$ can be made arbitrarily large by making
$\delta$ sufficiently small. Doing this,
we conclude that (ii) and (iii) of the
induction statement holds for $\nu$ satisfying  $n+p<\nu< n'$ where $n'$
is the next free return time.

We now turn to verifying (i) at the next free return time $n'$. Using
the previous derivative estimates and  (i) of Lemma \ref{mane} we get
after writing $n'=n+p+1+q$ that

$$
(f_{a,b}^{n'-1})'(f_{a,b})(c)\geq e^{2(n-1)^{2/3}+\frac{1}{10}p^{2/3}-\frac{8}{3}\sqrt{p+1}}\Cr{c:celb}e^{\Cr{kappa:outside}(q-1)},
$$
where $\Cr{c:celb}$ is an absolute constant only depending on
$a_0$. Arguing in different cases depending on the relative sizes of
$n$ and $q$, one can verify that (i) of the induction with $n$ replaced
by $n'$ holds. 

 Since $\Cr{c:celb}$ and $\Cr{c:lowerpdir}$ do not depend on $\delta$, while by making
 $\delta$ sufficiently small we can make $p$ as large as we want, we
 conclude using Lemma \ref{LengthBP} that

 $$\Cr{c:celb}\Cr{c:lowerpdir}e^r\cdot e^{-\sqrt{p+1}}\geq
 \Cr{c:celb}\Cr{c:lowerpdir}\exp\{\frac{1}{2}p^{2/3}-4\sqrt
 {p+1}\}\geq 1.
$$
 This
 proves~(iv) for $n'$.
\end{proof}

%From the above estimate it is also easy to conclude that (iv) of the
%induction also holds.
\medskip

We now delete the parameters which are mapped to

$$
(c-\Cr{c:celbstretched}\phantom{}^*e^{-\sqrt{n'}},c+ \Cr{c:celbstretched}\phantom{}^*e^{-\sqrt{n'}}),
$$
where $\Cr{c:celbstretched}\phantom{}^*
=\max(\Cr{c:celbstretched},1)$.  We conclude
    that (v) and (vi) of the induction also hold. This completes the proof
    of the induction step.

% \begin{equation}\label{si4}
% (f_{a,b}^{n+k})'(f_{a,b}(c))>C_{17}e^{\alpha n}e^{\kappa_5(n+k)},
% \end{equation}
% where
% \[
% C_{17}=C_7 C_{12}^2C_{13}^{-1}C_{11}C_{16}^Me^{-\kappa_0  M-\kappa_6}.
% \]
% Observe that if $k\le p$ then we get~\eqref{si4} from~\eqref{si3}.

% Now, while $C_{17}$ depends on $\delta$, according to
% Lemma~\ref{startupINT} we can make $n_0$ arbitrarily large without
% changing $\delta$. Since $n\ge n_0$, in such a way we can make
% $C_{17}e^{\alpha n}\ge 1$, so we get
% \[
% (f_{a,b}^{n+k})'(f_{a,b}(c))>e^{\kappa_5(n+k)}.
% \]
% Together with~(ii) and~(iii) this gives us~(ii) and~(iii) for $n$
% replaced by $n'$.

% To prove (iv) for $n'$, we only need to show that
% $(f_{a,b}^{n'-n})'(f_{a,b}^n(c))\ge 1$. To do it, since
% $n'-n=p+1+s_0$, we write
% \[
% (f_{a,b}^{n'-n})'(f_{a,b}^n(c))=(f_{a,b}^{p+1})'(f_{a,b}^n(c))
% \cdot(f_{a,b}^{s_0})'(f_{a,b}^{n+p+1}(c)).
% \]
% Thus, by Lemmas~\ref{deraftbd} and~\ref{mane}~(i), we get
% \[
% (f_{a,b}^{n'-n})'(f_{a,b}^n(c))\ge C_{14}e^r e^{-\beta p}\cdot
% C_{10}e^{\kappa_0 s_0}\ge C_{14}e^r e^{-\beta p}\cdot \Cr{c:celb}.
% \]
    We note that the proof also gives the information
\begin{equation}\label{sind1}
(f_{a,b}^{n'-n})'(f_{a,b}^n(c))\ge
 e^{\frac{1}{6}(n'-n)^{2/3}}.
 \end{equation}

%%{\tt Detta har lagts till ???}
 % By Lemma \ref{bounddist}, $r\ge p\kappa_5/4$, so
 % \begin{equation}\label{sind1}
 % (f_{a,b}^{n'-n})'(f_{a,b}^n(c))\ge
 % C_{10}C_{14}e^{(\kappa_5/4-\beta)p}.
 % \end{equation}

% Now assume that there is a good free return pair $(N,\Omega)$ such
% that $N\le n$, $\omega\subset\Omega$, and $n'\le (1+\iota)(N-1)+1$. Then
% \[
% (f_{a,b}^{n'-1})'(f_{a,b}(c))=(f_{a,b}^{N-1})'(f_{a,b}(c))\cdot
% (f_{a,b}^{n'-N})'(f_{a,b}^N(c)).
% \]
% By~(vi), the first factor on the right-hand side above is at least
% $e^{\kappa_7(N-1)}$, while the second one is by~(iv) for $n'$ (which
% we just proved), at least 1. Therefore,
% \[
% (f_{a,b}^{n'-1})'(f_{a,b}(c))\ge e^{\kappa_7(N-1)}=
% e^{\kappa_6(1+\iota)(N-1)}\ge e^{\kappa_6(n'-1)},
% \]
% so~(i) for $n'$ holds.

%%{\tt Detta har lagts till ???}
\begin{remark}\label{sind_rem}
 From \eqref{sind1} it immediately follows that 
 \begin{equation}\label{sind2}
 (f_{a,b}^{n'-1})'(f_{a,b}(c))\ge (f_{a,b}^{n-1})'(f_{a,b}(c))\cdot
e^{\frac{1}{6}p^3},
 \end{equation}
 which will be used later.
 We will also later use~\eqref{sind1}.
 \end{remark}
\begin{remark}\label{rem:growth}
Clearly, in Lemma \ref{le:growth}, $\omega$ can be replaced by any subinterval
$\omega'\subset\omega$. If we choose $\mu=n'$ and if
$\xi_{n'}(\omega',b)\subset I^*$, and $n+p+1\le\nu< \mu=n'$,
we can use Lemma~\ref{mane}~(i) to estimate $\inf_{a\in\omega}
(f_{a,b}^{n'-\nu})'(f_{a,b}^\nu(c))$ from below by
$\Cr{c:celb}e^{\Cr{kappa:outside}(n'-\nu)}$. Moreover by Lemma \ref{le:stepind} we
know that (iii) of the induction holds for $\nu<n'$. Therefore we
conclude from \eqref{deraceq} that \eqref{ax-dist-gen} holds with $q'=q_*$.
In such a way we get
\[
|\xi_{n'}(\omega',b)|\ge \frac{\Cr{c:celb}}{q_*}e^{\Cr{kappa:outside}(n'-\nu)}
|\xi_\nu(\omega',b)|.
\]
\end{remark}

\section{The global distortion Lemma}\label{gdl}

\begin{lemma}\label{globaldist}
There exists a constant $\Cr{c:alvesdist}$, such that if $a$ and $a'$ are two
parameter points, so that $a,a'\in \omega \in{\mathcal P}_n$, where
$n$ is a free return time and Induction Statement for $n$ (and all
smaller free return times) holds, then
\begin{equation}\label{estGlobal}
\frac{(f_{a,b}^{k})'(f_{a,b}(c))}{(f_{a',b}^{k})'(f_{a',b}(c))}\leq
\Cr{c:alvesdist}
\end{equation}
for all $k \leq n-1$.
\end{lemma}

\begin{proof}
Let us fix $k \leq n-1$. Set $t_0=1$ and let $\{t_j\}_{j=1}^m$ be the
free return times arranged in an increasing order. Here $m$ is defined
by the condition $t_{m-1} < k \leq t_m-1$, and we can assume that
$t_m=n$. Observe that for all free returns $t_j$ there is $r_j$ so
that $\xi_{t_j}(a,b), \xi_{t_j}(a',b) \in I_{r_j}$.

Note that $f'_{a,b}(x)=f'_{a',b}(x)$. Thus, using the Mean Value
Theorem, we can write the logarithm of the left hand side
of~\eqref{estGlobal} as
\begin{align*}
&\log \frac{(f_{a,b}^{k})'(f_{a,b}(c))}{(f_{a',b}^{k})'(f_{a',b}(c))}
= \log\frac{\prod_{j=1}^{k} f'_{a,b}(\xi_j(a,b))}
{\prod_{j=1}^{k} f'_{a,b}(\xi_j(a',b))}\\
& \leq\sum_{i=0}^{m-1}\left(\sum_{j=t_i}^{t_i+p_i}\left|
\frac{f_{a,b}''(\eta_j)}{f_{a,b}'(\eta_j)}\right|
|\xi_j(a,b)-\xi_j(a',b)|
+\log\frac{(f_{a,b}^{t_{i+1}-t_i-p_i-1})'(\xi_{t_i+p_i+1}(a,b))}
{(f_{a',b}^{t_{i+1}-t_i-p_i-1})'(\xi_{t_i+p_i+1}(a',b))}\right)
\end{align*}
for some $\eta_j$ between $\xi_j(a,b)$ and $\xi_j(a',b)$, where $p_i$
is the corresponding bound time and $p_0=-1$. We will denote the
first sum in the parenthesis above by $S_i'$ and the second term in
parenthesis by $S_i''$. Note that the sum $S_0'$ is empty.

By~\eqref{secondb} and~\eqref{sec_der} we get
\begin{equation}\label{eq:critbev}
\left|\frac{f_{a,b}''(\eta_j)}{f_{a,b}'(\eta_j)}\right|\leq
\frac{80}{\Cr{c:lder}|\eta_j-c|}.
\end{equation}

Set $\sigma_i=|\xi_{t_i}(a,b)-\xi_{t_i}(a',b)|$. We claim that the sum
$S_i'$ can be estimated from above by a constant times $\sigma_i
e^{r_i}$.

First we note that by~\eqref{eq:critbev}, the first term of $S_i'$ can
be estimated by $80\sigma_i/(\Cr{c:lder} e^{-r_i})$.

For the remaining terms we introduce the reference interval
$\Omega_{t_i}=I_{r_i+1}$ and intervals $\Omega_{t_i+\nu}=
f_{a,b}^\nu(\Omega_{t_i})$, $\nu=0,1,2,\dots,p_i$. We have
\begin{align*}
\xi_{t_i+1}(a,b)-\xi_{t_i+1}(a',b)&=
\left(f_{a,b}(\xi_{t_i}(a,b))-f_{a,b}(\xi_{t_i}(a',b)\right)\\
&+\left(f_{a,b}(\xi_{t_i}(a',b)-f_{a',b}(\xi_{t_i}(a',b)\right)\\
&=f_{a,b}'(y)(\xi_{t_i}(a,b)-\xi_{t_i}(a',b))+(a-a')
\end{align*}
for some $y$ between $\xi_{t_i}(a,b)$ and $\xi_{t_i}(a',b)$.
Furthermore, $|\Omega_{t_i+1}|=f_{a,b}'(y')|\Omega_{t_i}|$ for some
$y'\in\Omega_{t_i}$.

We get
\begin{equation}\label{xiOmega}
\frac{|\xi_{t_i+1}(a,b)-\xi_{t_i+1}(a',b)|}{|\Omega_{t_i+1}|}
=\frac{f_{a,b}'(y)}{f_{a,b}'(y')}\cdot\frac{\sigma_i}
{|\Omega_{t_i}|}\pm\frac{|a-a'|}{f_{a,b}'(y')|\Omega_{t_i}|}.
\end{equation}
By the Mean Value Theorem
\[
\frac{|a-a'|}{f_{a,b}'(y)\sigma_i}=\frac1{f_{a,b}'(y)\partial_a
\xi_{t_i}(a'',b)}
\]
for some $a''$ between $a$ and $a'$. By~\eqref{derax} and the
induction statement~(iii),
\[
\partial_a\xi_{t_i}(a'',b)\ge(f_{a'',b}^{t_i-1})'(f_{a'',b}(c))
\ge \Cr{c:celbstretched}e^{\sqrt{t_i-1}}.
\]
By~\eqref{secondb} and the induction statement~(v),
\[
f_{a,b}'(y)\ge \Cr{c:lder}\Cr{c:cba-sqrt}^2e^{-2\sqrt{t_i}}.
\]
Therefore we get
\[
\frac{|a-a'|}{f_{a,b}'(y)\sigma_i}\le\frac{e}
{\Cr{c:lder}\Cr{c:cba-sqrt}^2\Cr{c:celbstretched}}e^{(2\sqrt{t_i}-(t_i)^{2/3})}.
\]
Since $t_i$ can be made as large as we want
(because $t_i\ge n_0$), we get $|a-a'|<f_{a,b}'(y)\sigma_i/2$. Therefore
from~\eqref{xiOmega} we get
\[
\frac12\frac{f_{a,b}'(y)}{f_{a,b}'(y')}\cdot\frac{\sigma_i}
{|\Omega_{t_i}|}<\frac{|\xi_{t_i+1}(a,b)-\xi_{t_i+1}(a',b)|}
{|\Omega_{t_i+1}|}<2\frac{f_{a,b}'(y)}{f_{a,b}'(y')}\cdot
\frac{\sigma_i}{|\Omega_{t_i}|}.
\]

Since $y,y'\in I_{r_i}\cup I_{r_i+1}$, we have
$|y-c|\in[e^{-r_i-2},e^{-r_i}]$, and the same holds for $y'$.
Therefore, by~\eqref{secondb}, we get
\[
\frac{f_{a,b}'(y)}{f_{a,b}'(y')}\le\frac{2-2b+\Cr{c:uder}e^{-2r}}
{2-2b+\Cr{c:lder}e^{-2r-4}}.
\]
The right-hand side above is a weighted average between
$(2-2b)/(2-2b)=1$ and $(\Cr{c:uder}e^{-2r})/(\Cr{c:lder}e^{-2r-4})=\Cr{c:uder}e^4/\Cr{c:lder}>1$, so
it is smaller than $\Cr{c:uder}e^4/\Cr{c:lder}$. Since we can switch $y$ and $y'$, we
get
\[
\frac{\Cr{c:lder}}{\Cr{c:uder}e^4}<\frac{f_{a,b}'(y)}{f_{a,b}'(y')}<
\frac{\Cr{c:uder}e^4}{\Cr{c:lder}}.
\]
In such a way we get the following inequality with
$\Cl{c:distparder}=2\Cr{c:uder}e^4/\Cr{c:lder}$.
\begin{equation}\label{eq:compbound1}
\Cr{c:distparder}^{-1}\frac{\sigma_i}{|\Omega_{t_i}|}\leq
\frac{|\xi_{t_i+1}(a,b)-\xi_{t_i+1}(a',b)|}{|\Omega_{t_i+1}|}\leq
\Cr{c:distparder} \frac{\sigma_i}{|\Omega_{t_i}|}.
\end{equation}

Now we want to estimate from above
\[
\frac{|\xi_{t_i+{\nu}}(a,b)-\xi_{t_i+{\nu}}(a',b)|}
{|\Omega_{t_i+{\nu}}|}.
\]
The numerator can be estimated as follows:
\begin{align}\label{eq:termnu}
|\xi_{t_i+{\nu}}(a,b)-\xi_{t_{i+{\nu}}}(a',b)|
&\leq |f_{a,b}^{{\nu}-1}(f^{t_i+1}_{a,b}(c))
-f_{a,b}^{{\nu}-1}(f^{t_i+1}_{a',b}(c))|\\
&+|f_{a,b}^{\nu}(f^{t_i}_{a',b}(c))
-f_{a',b}^{\nu}(f^{t_i}_{a',b}(c))|. \nonumber
\end{align}
By a similar argument as in the proof of Lemma \ref{LengthBP},
in particular estimating $|a-a'|$ as in that lemma
we obtain that 

$$
|f_{a,b}^{\nu}(f^{t_i}_{a',b}(c))
-f_{a',b}^{\nu}(f^{t_i}_{a',b}(c))|\leq \frac{1}{2}|\xi_{t_i+{\nu}}(a,b)-\xi_{t_{i+{\nu}}}(a',b)|.
$$

Therefore we get the estimate
% (taking into
% account also~\eqref{fterm}) we get
\begin{align}\label{termnu1}
|\xi_{t_i+{\nu}}(a,b)-\xi_{t_{i+{\nu}}}(a',b)|
&\leq 2|f_{a,b}^{{\nu}-1}(f^{t_i+1}_{a,b}(c))
-f_{a,b}^{{\nu}-1}(f^{t_i+1}_{a',b}(c))|\\
&=2(f^{{\nu}-1}_{a,b})'(y)|\xi_{t_i+1}(a,b)-
\xi_{t_i+1}(a',b)|,\nonumber
\end{align}
where $y$ is between $\xi_{t_{i+1}}(a,b)$ and $\xi_{t_{i+1}}(a',b)$ 
(using the Mean Value Theorem).
Again by the same theorem there is $y''\in\Omega_{t_i+1}$ such that
\[
|\Omega_{t_i+{\nu}}|= (f^{{\nu}-1}_{a,b})'(y'')|\Omega_{t_i+1}|.
\]
By Lemma \ref{bounddist}, $(f^{{\nu}-1}_{a,b})'(y)/
(f^{{\nu}-1}_{a,b})'(y'')\le \Cr{c:bddist}^2$, so we get
\begin{equation}\label{eq:compbound2}
\frac{|\xi_{t_i+{\nu}}(a,b)-\xi_{t_{i}+{\nu}}(a',b)|}
{|\Omega_{t_i+{\nu}}|}\leq 2\Cr{c:bddist}^2
\frac{|\xi_{t_i+1}(a,b)-\xi_{t_i+1}(a',b)|}{|\Omega_{t_i+1}|}.
\end{equation}

Let us consider the interval $\omega_i\in\calp_{t_i}$ containing
$\omega$ and denote its midpoint by $\ahat_i$. Then by~\eqref{BC},
\begin{equation}\label{61a}
|\xi_{t_i+\nu}(\atil,b)-\xi_\nu(\ahat_i,b)|\le e^{-4\sqrt{\nu}}
\end{equation}
for all $\atil\in\omega_i$ and $\nu\le p_i$.

We claim that for all $\atil\in\omega_i$ and $\nu\le p_i$ we have
\begin{equation}\label{compdist}
|\xi_{t_i+\nu}(\atil,b)-c|\geq\frac12|\xi_\nu(\ahat_i,b)-c|.
\end{equation}

There is an integer $\nu_0$ such that
\begin{equation}\label{bound-alpha-beta}
\frac12\cdot \Cr{c:cba-sqrt}e^{-\sqrt{ \nu}}\geq e^{-4\sqrt{\nu}}
\end{equation}
for all $\nu\geq\nu_0$. Note that $\nu_0$ depends only on $\Cr{c:cba-sqrt}$,
which is independent of $\delta$. Therefore, we may assume that $\delta$ is so
small that
\begin{equation}\label{61b}
2\delta\cdot 4^{\nu_0}\le\frac12\cdot \Cr{c:cba-sqrt}e^{-\sqrt{ \nu_0}}.
\end{equation}
Moreover, by Induction Statement~(v),
\begin{equation}\label{61d}
|\xi_\nu(\ahat_i,b)-c|\ge \Cr{c:cba-sqrt}e^{-\sqrt{ \nu}}.
\end{equation}

Consider $\nu\le p_i$. If $\nu\ge\nu_0$, then by~\eqref{61a},
\eqref{bound-alpha-beta} and~\eqref{61d},  we get
\begin{equation}\label{61c}
|\xi_{t_i+\nu}(\atil,b)-\xi_\nu(\ahat_i,b)|\le
\frac12|\xi_\nu(\ahat_i,b)-c|.
\end{equation}
Therefore
\[
|\xi_{t_i+\nu}(\atil,b)-c|\ge |\xi_\nu(\ahat_i,b)-c|-
|\xi_{t_i+\nu}(\atil,b)-\xi_\nu(\ahat_i,b)|\ge
\frac12|\xi_\nu(\ahat_i,b)-c|
\]
and~\eqref{compdist} follows.

If $\nu<\nu_0$ then by~\eqref{derxbd} and~\eqref{derabd}
\begin{align*}
|\xi_{t_i+\nu}(\atil,b)-\xi_\nu(\ahat_i,b)|
&\le |\xi_{t_i+\nu}(\atil,b)-\xi_\nu(\atil,b)|+
|\xi_\nu(\atil,b)-\xi_\nu(\ahat_i,b)|\\
&\le 4^\nu|\xi_{t_i}(\atil,b)-c|+4^\nu|\atil-\ahat_i|\le
4^\nu(\delta+|\omega_i|).
\end{align*}
By Lemma~\ref{startupINT}~(b) for $j=n_0$ and by making $n_0$
sufficiently large, we get $\partial_a\xi_j(a,b)\ge 1$. Therefore
\[
|\omega|\le |\xi_{n_0}(\omega,b)|\le\delta.
\]
Thus, $|\omega_i|\le\delta$, and we get
\[
|\xi_{t_i+\nu}(\atil,b)-\xi_\nu(\ahat_i,b)|\le 2\delta\cdot 4^\nu.
\]
Together with~\eqref{61b} and~\eqref{61d} we get also in this
case~\eqref{61c}, and~\eqref{compdist} follows.

Now for each $\nu$ we choose $\tilde{a}_\nu$ between $a$ and $a'$, so
that $\xi_{t_i+\nu}(\tilde{a}_\nu,b)=\eta_{t_i+\nu}$. Thus,
by~\eqref{compdist} and~\eqref{61d},
\begin{equation}\label{61e}
|\eta_{t_i+\nu}-c|=|\xi_{t_i+\nu}(\tilde{a}_\nu)-c|\geq
\frac12|\xi_{\nu}(\ahat_i,b)-c|\ge \frac12 \Cr{c:cba-sqrt}e^{-\sqrt{\nu}}.
\end{equation}

By~\eqref{eq:critbev} we have
\begin{align}\label{61f}
\sum_{j=t_i+1}^{t_i+p_i} \left|\frac{f_{a,b}''(\eta_j)}
{f_{a,b}'(\eta_j)}\right|&|\xi_j(a,b)-\xi_j(a',b)|\\
&\leq\sum_{{\nu}=1}^{p_i}\frac{80}{\Cr{c:lder}}\cdot\frac{|\Omega_{t_i+\nu}|}
{|\eta_{t_i+\nu}-c|}\cdot\frac{|\xi_{t_i+{\nu}}(a,b)
-\xi_{t_i+{\nu}}(a',b)|}{|\Omega_{t_i+{\nu}}|}\nonumber
\end{align}
By the definition of bound periods and the definition of $\Omega_{t_i+\nu}$,
we have
\[
|\Omega_{t_i+\nu}|\le e^{-4\sqrt{\nu}}.
\]
Moreover,
\[
|\Omega_{t_i}|=e^{-r_i-1}-e^{-r_i-2}.
\]
Substituting those two inequalities,~\eqref{61e},
\eqref{eq:compbound2} and~\eqref{eq:compbound1} into the right-hand
side of~\eqref{61f}, we get
\[
\sum_{j=t_i+1}^{t_i+p_i} \left|\frac{f_{a,b}''(\eta_j)}
{f_{a,b}'(\eta_j)}\right||\xi_j(a,b)-\xi_j(a',b)|\le
\sum_{{\nu}=1}^{p_i}\Cl{c:globalboundterm}\cdot\frac{e^{-4\sqrt{\nu}}}
{e^{-\sqrt {\nu}}}\cdot\frac{\sigma_i}{e^{-r_i}}
\]
for some constant $\Cr{c:globalboundterm}$. This implies that
there is a constant $\Cl{c:globalboundterm1}$, such that
\[
\sum_{j=t_i+1}^{t_i+p_i} \left|\frac{f_{a,b}''(\eta_j)}
{f_{a,b}'(\eta_j)}\right||\xi_j(a,b)-\xi_j(a',b)|\le
\Cr{c:globalboundterm1}\frac{\sigma_i}{e^{-r_i}}.
\]
Together with the estimate on the first term of $S_i'$, that we
obtained long ago, we get a constant $\Cl{c:longdistbound}$ such that
\begin{equation}\label{61g}
S_i'\le \Cr{c:longdistbound}\frac{\sigma_i}{e^{-r_i}}.
\end{equation}
Note that $\Cr{c:longdistbound}$ depends on $\Cr{kappa:mane}$, but not on $\delta$.

To estimate $S_i''$, we use Lemma~\ref{le:outside}, and get
immediately
\[
S_i'' \leq \Cr{c:freedist} \frac{\sigma_{i+1}}{\delta}.
\]
However, $\delta>e^{-r_{i+1}}$, so
\begin{equation}\label{61h}
S_i''\leq \Cr{c:freedist} \frac{\sigma_{i+1}}{e^{-r_{i+1}}}.
\end{equation}
This estimate also applies to $S_0''$.

By Lemma~\ref{le:growth} applied to a subinterval $\omega'=[a,a']$ (or
$[a',a]$) of $\omega$ (we can do it by Remark~\ref{rem:growth})
and~\eqref{sind1}, see Remark~\ref{sind_rem}, we get
\[
\sigma_{i+1}\ge\frac1{q_*}\inf_{\atil\in\omega'}
(f_{\atil,b}^{t_{i+1}-t_i})'(f_{\atil,b}^{t_i}(c))\cdot\sigma_i
\ge\frac{\Cr{c:celb}\Cr{c:lowerpdir}}{q_*}e^{(p_i^{2/3}-4\sqrt{p_i+1})}
\cdot\sigma_i.
\]
As we already noticed in the proof of Lemma~\ref{le:stepind}, by
taking $\delta$ sufficiently small we can make $p_i$ as large as we
need and we may assume that
\[
\frac{\Cr{c:celb}\Cr{c:lowerpdir}}{q_*}\exp\{p_i^{2/3}-4\sqrt{p_i+1}\}\ge 2,
\]
and therefore we get
\begin{equation}\label{61i}
\sigma_{i+1}\ge2\sigma_i.
\end{equation}

Now we are ready to estimate the logarithm of the left-hand side
of~\eqref{estGlobal}, which is less then or equal to
$\sum_{i=0}^{m-1}(S_i'+S_i'')$. By~\eqref{61g} and~\eqref{61h}, we get
\[
\sum_{i=0}^{m-1}(S_i'+S_i'')\le(\Cr{c:longdistbound}+\Cr{c:freedist})\sum_{i=0}^m
\frac{\sigma_i}{e^{-r_i}}.
\]
Rearrange the sum $\sum_{i=0}^m\sigma_i/e^{-r_i}$ and group it
according to the values of $r_i$. Set $W_k=\{i\in[1,m]:r_i=k\}$.
Consider $k$ such that $W_k$ is nonempty. Then we can write
$W_k=\{i_s<i_{s-1}<\dots<i_0\}$, and by~\eqref{61i}, we have
$\sigma_{i_j}\le\sigma_{i_0}/2^j$. Thus,
\[
\sum_{i\in W_k}\frac{\sigma_i}{e^{-r_i}}\le2\frac{\sigma_{\mu_k}}
{e^{-k}},
\]
where $\mu_k$ is the largest element of $W_k$. However,
$\sigma_{\mu_k}$ is the length of an interval which is contained in
the union of 3 subintervals of $I_k$, and the length of each of those
subintervals is $|I_k|/k^2$. Moreover, $|I_k|<e^{-k}$. Thus,
\begin{equation}\label{61j}
\sum_{i\in W_k}\frac{\sigma_i}{e^{-r_i}}\le\frac6{k^2}.
\end{equation}
If $W_k$ is empty, then of course~\eqref{61j} also holds. In such a
way we get
\[
\sum_{i=0}^{m-1}(S_i'+S_i'')\le6(\Cr{c:longdistbound}+\Cr{c:freedist})\sum_{k=1}^\infty
\frac1{k^2}.
\]
The right-hand side of the above inequality is finite, so we can
denote its exponential by $\Cr{c:alvesdist}$ and then~\eqref{estGlobal} holds.
\end{proof}

\begin{lemma}\label{le:globaldistpar}
There exists a constant $\Cl{c:pardist}$, such that if $a$ and $a'$ are two
parameter points, so that $a,a'\in \omega \in{\mathcal P}_n$, where
$n$ is a free return time and Induction Statement for $n$ (and all
smaller free return times) holds, then
\begin{equation}\label{estGlobal1}
\frac{\partial_a f_{a,b}^{k}(c)}{\partial_a f_{a',b}^{k}(c)}\leq
\Cr{c:pardist}
\end{equation}
for all $k \leq n$.
\end{lemma}

\begin{proof} The lemma follows immediately from
Lemma~\ref{globaldist}, Corollary~\ref{dera_cor} and Induction
Statement~(iii).
\end{proof}

\section{Part I of the proof of Theorem A}\label{slargedev2}

In this section we prove a proposition, which is an essential part of the
proof of Theorem A, and is stated as follows.

% \begin{proposition}\label{MainThmcritCE}
% Let $a=a_0$ be an MT-parameter for $f_{a}$ and let $\varepsilon>0$ be
% given. There is a function $\eta(\varepsilon)\to 0$ as  such that if $\omega_0$
% is a parameter interval such that
% $\omega_0\subset(a_0-\varepsilon,a_0+\varepsilon)$  and such that
% $I_{r,\ell}\subset \xi_{n_0}(\omega_0)\subset I_{r,\ell}^+$ and such
% that induction assumptions (i)--(vi) are satisfied for $n=n_0$.
% Then there is a subset $E_0$ so that $|E_0|\geq
% (1-\eta(\varepsilon))|\omega_0|$, $C=C(a_0)$
% and $\kappa_5=\kappa_5(a_0)>0$ so that   

% \begin{equation}\label{cedsm1}
% (f_{a}^n)'(f_{a}(c)) \geq C e^{\kappa_5 n},\qquad \forall n\geq
% 0\qquad \forall a\in E_0.
% \end{equation}
% \end{proposition}

\begin{proposition}\label{MainThmfixedbCE}

Let $a=a_0$ be an MT-parameter for $f_{a}$ and let $\varepsilon>0$ be
given. There is a function $\eta(\varepsilon)\to 0$ and a function
$b_0(\varepsilon)\to 1$ as $\varepsilon\to0$ such that if
$b_0(\varepsilon)<b<1$, if $\omega_0$
is a parameter interval such that 

\begin{equation}\label{eq:inclusion}
  \omega_0\subset(a_0-\varepsilon,a_0-\varepsilon^2)\cup(a_0+\varepsilon^2,a_0+\varepsilon),
\end{equation}
such that
$I_{r,\ell}\subset \xi_{n_0}(\omega_0,b)\subset I_{r,\ell}^+$ and such
that induction assumptions (i)--(vi) are satisfied for $n=n_0$.
Then there is a set $\tilde{E}_b\subset\omega_0$ so that $|\tilde{E}_b|\geq
(1-\eta(\varepsilon))|\omega_0|$, $C=C(a_0)$
and $\hat{\kappa}=\hat{\kappa}(a_0)>0$ so that   

\begin{equation}\label{cedsm2}
(f_{a,b}^n)'(f_{a,b}(c)) \geq C e^{\hat{\kappa} n},\qquad \forall n\geq
0\qquad \forall a\in \tilde{E}_b.
\end{equation}
\end{proposition}

Note that the assumptions of Proposition 
\ref{MainThmfixedbCE} is satisfied by Lemma \ref{startupINT}.

This, together with Proposition \ref{hyper1} in Section 8 immediately lead to the
following 

\begin{corollary} The set $E$ of parameters for which the double 
standard map is uniformly expanding accumulates on the MT points 
$(a_0,1)$ in the parameter space.
\end{corollary}
 
However we will need more general formulation of the
propositions given above in order to prove Theorem A.

% Although for the family of Double Standard Maps \eqref{DSM} with
% $b=1$, $c=1/2$ is a critical point, for $b<1$ we have that $c$ is an
% inflection point. We begin this section proving an analogue to the
% Collet-Eckmann condition for the image of the inflection point.

% The following proposision is a preliminary result which can be
% thought of as a step toward Theorem~B and a simplified version of
% it.
%
%\begin{proposition}\label{pr:nonempty}
%Let $E$ be the set of parameters $(a,b)$ such that $f_{a,b}$ is
%uniformly expanding (note that $E$ is an open set). Then $E$
%accumulates on the MT parameters of the family $\{f_{a,1}\}$.
%
%More precisly, if $a_0$ is an MT parameter and $\varepsilon$
%is given, then there is $b_0=b_0(a_0,\varepsilon)$ so that
%\begin{equation}\label{nonempty}
%E\cap\{(a,b):a_0-\varepsilon\leq a \leq a_0+\varepsilon
%\,,b_0<b<1\}\neq\emptyset.
%\end{equation}
%\end{proposition}

\medskip
The proofs will be based on the induction formulated in Section
\ref{sec:bound-free-ess}. In the critical case $b=1$, which we are not
treating in detail, the remaining
parameter set is of positive measure, while in the non-critical case $b<1$
the remaining parameter set is a finite union of intervals.

We first discuss the parameter deletion due to the (BA) assumption.

If $n$ is a free essential return time for a partition element
$\omega=(a,a')$ of a partition ${\mathcal P}_{n''}$, where $n''$ is
the essential free return immediately before $n$.

At each time we may have to omit a fraction of the parameter interval
because of (BA). Assume that the previous free return occurred in the
interval $I_{r'',\ell}$. Its length is $\frac{c}{r''^2}|I_{r''}|$,
$1\leq c\leq 3$. By
the (BA) assumption applied to time $n''$, we have
$$
e^{-r''}\geq  e^{-\sqrt{ n''}}
$$
Since $n-n''$ has a minimal length with estimate $n-n''\geq
C\log(1/\delta)$, where $C$ is a constant only depending on $f_{a_0}$.
Not also that $r''\leq \sqrt{n}$.

During the bound period the interval ${\mathcal K}_{r''+1}=(c,c+e^{-r''-1)})$ of size
$e^{-r''-1}$ is increased to size $e^{-\sqrt{p''+1}}$, where $p''\leq
8(r'')^{3/2}$ by Lemma \ref{LengthBP}. Our present interval is
of length $\frac{c'}{r''^2}|I_{r''}|$, $1\leq c' \leq 3$.

For $a=a'$ the size of $f_{a',b}({\mathcal K}_{r''+1})$ can be estimated by 
formula \eqref{thirdb} as follows
$$
|x-c|(2-2b+\Cr{c:lder}(x-c)^2)\leq |f_{a',b}(x)-f_{a',b}(c)|\leq|x-c|(2-2b+\Cr{c:uder}(x-c)^2).
$$
By inserting $x=c+e^{-r''-1}$ we obtain an estimate for
$|f_{a,b}({\mathcal K}_{r''+1})|$
as follows:
\begin{equation}
e^{-r''-1}(2-2b+\Cr{c:lcr}e^{-2r''-2})\leq |f_{a,b}({\mathcal K}_{r''+1})| \leq
e^{-r''-1}(2-2b+\Cr{c:ucr}e^{-2r''-2}).
\end{equation}
%An analogous estimate holds for $f_{a',b}({\mathcal K}_r)$.
For the image of $\omega$ at time $n''+1$ we obtain the estimate
\begin{equation}\label{xinprime1}
|\xi_{n''+1}(a,b)-\xi_{n''+1}(a',b)|=f
_{a,b}'(y)\cdot|\xi_{n''}(a,b)-\xi_{n''}(a',b)| \pm |a-a'|
\end{equation}
Here $y\in I_r$ so it follows from \eqref{secondb} that $|a-a'|$ can
be estimated by the first term as in the estimate of \eqref{xiOmega}
and we obtain
$$
2-2b+\Cr{c:lder}e^{-2r''-2}<f'_{a,b}(y)<2-2b+\Cr{c:uder}e^{-2r''}.
$$

By the definition of a free return we also have the estimate
$$
\frac{1}{r''^2}e^{-r''}\leq |\xi_{n''}(a,b)-\xi_{n''}(a',b)|\leq
\frac{3}{r''^2}e^{-r''}.
$$
By Lemma \ref{bounddist} (the bound distortion lemma) and comparison
with the orbit of ${\mathcal K}_{r''}$ the size of
$|\xi_{n''+p''+1}(a,b)-\xi_{n''+p''+1}(a',b)|$ has the lower bound.
$$
\frac{1}{\Cr{c:bddist}}\cdot\frac{1}{r''^2}e^{-4\sqrt{p''+1}} \geq
\frac{1}{\Cr{c:bddist}'}\cdot\frac{1}{r''^2}e^{-(\sqrt{8(r'')^{3/2}}}\geq
\frac{1}{\Cr{c:bddist}'}e^{-8^{3/2}\cdot n^{3/8}} .
$$
We have again used that $\delta$ may be chosen arbitrarily small. 
Using \eqref{61i}, Lemma \ref{mane} and Lemma \ref{le:growth} and
it follows that the relative fraction to be
deleted is at most
\begin{equation}\label{fraction0}
\Cr{c:bddist}'^{-1}\Cr{c:celb}\frac{1}{q_*}\frac {e^{-\sqrt{n}}}{e^{-8^{3/2}\cdot n^{3/8}}} < e^{-\frac{1}{2}\sqrt{n}},
\end{equation}
since $n\geq n_0$ which 
at each time $n$ we in principle may to have to do such a deletion.
The remaining fraction of the parameter interval can then be estimated
from below as
\begin{equation}\label{fraction}
\geq \prod_{n=N_0}^\infty\left(1- e^{-\frac{1}{2} \sqrt{n}}\right)
\end{equation}
Note that this is arbitrarily close to 1 as $N_0(\varepsilon)\to\infty$
as $\varepsilon\to 0$.

\begin{proof}[Outline of proof of Proposition~\ref{MainThmfixedbCE}]
The proof of Proposition~\ref{MainThmfixedbCE} is based on the
induction. Note that the Cantor Set construction can be stopped at a finite
stage $\hat{N}$, which is defined by the relation
$$
2-2b\geq \Cr{c:lder} e^{-2\sqrt{\hat{N}}}.
$$
After this time the term $2-2b$ of equation \ref{thirdb} dominates in
the derivative and we conclude that for all $a\in\tilde{E}_b$ there is
a constant $C>0$ so that 
\begin{equation}
(f^n_{a,b})'(f_{a,b}(c))\geq Ce^{\tilde{\kappa}n}\qquad \forall n\geq 0.
\end{equation}  

A more general result with a more detailed proof is given in Proposition \ref{hyper1}. 

\end{proof}

\medskip

{\em Outline of the proof of Proposition \ref{prop:stretched:exp}}. We proceed as in the proof of Theorem A. In this case the time $\hat{N}$, after which
the linear term $2-2b$ dominates in the derivative does not exist and
the induction proceeds to infinite time.  We now have to use the large
deviation argument of \cite{BC2}. The main idea is that you delete
parameters for which the critical orbits spend to much fractions of
the time  recovering the derivative loss.
from returns to $(c-\delta^2,c +\delta^2)$. However an estimate
similar to \eqref{fraction} is still valid. We do not give the full details.

\section{Part II of the proof of Theorem A --- the  uniform expansion}\label{improvedsel}
In this section we consider $b<1$,  and  we construct 
a non-empty union of open intervals  $\hat{E}_b \supset \tilde{E}_b$ so that
for $a\in \hat{E}_b$   there is an integer  $N$ so that, $f_{a,b}^N$ is uniformly expanding. This is
formulated in Proposition \ref{hyper1}. The set $\hat{E}_b$ is
obtained by stopping the construction of the parameter set
$\tilde{E}_b$ of  Proposition \ref{MainThmfixedbCE} at a finite stage.

\smallskip
Let us outline the main idea of the proof of the uniform expansion.
We will heavily use that the fact that $d=2-2b>0$, i.e. that the inflexion point is
non-critical. In the case the starting point $x$ is outside the return
interval $I^*$ we can uses (i) of Lemma \ref{mane} to conclude that 
if $x, f_{a,b} (x),\dots, f_{a,b}^{n-1}(x) \notin I^*$, and
$f_{a,b}^n(x)\in I^*$ then
\[
(f_{a,b}^n)'(x)\geq \Cr{c:celb}e^{\Cr{kappa:outside}n}.
\]
Here it is important that the constant $\Cr{c:celb}$ does not depend on
$\delta$.

At the return time $n$ we have a derivative loss but this derivative
loss is compensated during the bound period by Lemma \ref{dercompBP}.
Since $p\to\infty$ as $\delta\to 0$,  we can  make the factor $e^{p^{2/3}}$ compensate $C_2/{C^*}$ by
making $\delta$ sufficiently small. We also use that the derivative 
of $f_{a,b}$ is bounded below by $f'_{a,b}(\frac12)=2-2b$, and we
will also denote this number by $d$.
\medskip

%\begin{theorem} There is a set $A_1\subset A$ of non-empty interior
%  in the parameter space $[0,1)\times(\frac12,1)$ so that for all
%  $(a,b)\in A$, $f=f_{a,b}$ is hyperbolically expanding, i.e. there is
%  $N=N(a,b)$ and $\lambda>1$ so that
%$$
%Df^N(x)\geq\lambda^N.
%$$
%\end{theorem}

We state this result as follows.

\begin{proposition}\label{hyper1}
Let $a=a_0$ be a MT parameter. Then if $b_0=b_0(a_0)<1$ is
sufficiently close to 1 then for all $b\in(b_0,1)$ there is a set
$\hat{E}_b$, which is a finite union of intervals
$\{\omega_j\}_{j=0}^{J_0}$ so that for $a\in \omega_j$, there is an
integer $M_j$ so that for all $x\in {\mathbb T}$,
\begin{equation}\label{hyperbolicity}
(f_{a,b}^{M_j})'(x)\geq \lambda_j>1.
\end{equation}
\end{proposition}

\begin{proof}
The proof is really the same as the proof of Proposition \ref{MainThmfixedbCE} initially.

  % The new aspect is that it is possible to stop the
  % parameter selection at a finite stage, when the remaining parameter
  % set is a finite union of intervals.

As before, we carry out the construction only until time $\hat{N}$. Here
$\hat{N}$ is the the smallest integer $\hat{N}$ satisfying
$$
e^{-\sqrt{ \hat{N}}}\leq d.
$$
At time $\hat{N}$ we have a partition ${\mathcal P}_{\hat{N}}$ consisting of
finitely many intervals $\{\omega_j\}_{j=1}^{M_{\hat{N}}}$.

We now aim to prove that the hyperbolicity statement \eqref{hyperbolicity} is true.

We first recall the two outside expansion statements of
Lemma~\ref{mane}. Suppose that $(a,b)\in{\mathcal N}$ and chose
$I=I^*=(c-\delta,c+\delta)$ in that lemma. Then the
following holds.
\begin{itemize}
\item[1)] If $x,\ f_{a,b}x,\dots, f_{a,b}^{n-1}x\not\in I^*$ and $f_{a,b}^n x\in I^*$,
$$
(f_{a,b})'(x)\geq \Cr{c:celb}e^{\Cr{kappa:outside} n}.
$$
\item[2)] There is an integer $M$ so that if $x,
  f_{a,b}x,\dots,f_{a,b}^{M-1}x\not\in I^*$ then
$$
(f_{a,b}^M)'(x)\geq e^{\Cr{kappa:outside} M}.
$$
\end{itemize}
\end{proof}

Let us define $R_0$ as the smallest integer $R_0$ satisfying $e^{-2R_0}\leq e^{-\sqrt {\hat{N}}}$,
i.e. $R_0$ corresponds to the $r$ where the square term in the
expression for the derivative is of the same size as the constant term $d=2-2b$. The
bound period $p(x),\ x\in(c-e^{-R_0},c+e^{-R_0})$ is chosen to be the
infimum of the bound period for $y\in I_{\pm R_0}$.

We also know by \eqref{ih2} and Lemma \ref{bounddist} that
$$
(f_{a,b}^p)'(x)\geq \frac{1}{\Cr{c:bddist}}e^{ p^{2/3}},\qquad x\in I_{\pm r},
$$
and it holds as well that
$$
(f_{a,b}^p)'(x)\geq \frac{1}{\Cr{c:bddist}} e^{ p^{2/3}},\qquad x\in (c-e^{-R_0},c+e^{-R_0}).
$$
Introducing $\Cl[Kappa-const]{kappa:trunc}$ as
$$
\Cr{kappa:trunc}=\frac{1}{2}\min_{r_\delta\leq |r|\leq R_0}\min_{x\in I_r} \frac{1}{p(x)^{1/3}},
$$
we can in all cases write these estimates as
\begin{equation}\label{eq:lowerbdp}
(f_{a,b}^p)'(x)\geq e^{\Cr{kappa:trunc} p}.
\end{equation}

The factor $\frac12$  is here  used to absorb the constant $\Cr{c:bddist}$. 

Let us in the following use the notation $\hat{I}_{R_0}$ for
the union of $(c-e^{-R_0-1},c+e^{-R_0-1})$ and the previously defined
$I_{-R_0}$ and $I_{R_0}$. The idea is that the derivative recovery has
the same estimate for these three (original) intervals since
$(f_{a,b})'(x)\sim d=2-2b$ in $\hat{I}_{R_0}$ and the bound period is
defined in terms of ${I}_{R_0}$.

Divide the set ${\mathbb T}\setminus I^*$ into several pieces.

We first consider the set
$$
X_M=\{x:x,f_{a,b}x,\dots,f_{a,b}^Mx\not\in I^*\}.
$$
For $x\in X_M$, hyperbolicity is valid by Lemma~\ref{mane}, (ii):
$$
(f_{a,b}^M)'(x)\geq e^{\Cr{kappa:outside} M}.
$$

We also introduce the sets
$$
X_k=\{x:x,\dots,f_{a,b}^{k-1}x\not\in I^* \text{ but } f_{a,b}^kx\in I^*\}, \qquad
1\leq k \leq M-1.
$$

Pick a $k\geq 1$. Now write  the set
$$
X_k=\bigcup_{r_\delta\leq |r|\leq R_0}X_{k,r},
$$
where $X_{k,r}=\{x\in X_k: f_{a,b}^kx\in I_r\}$, $|r_\delta|\leq |r|<R_0$
and
$$
X_{k,\pm R_0}=\{x\in X_k:f^k_{a,b}x\in (c-e^{-R_0},c+e^{-R_0}\}.
$$

We then know that for $x\in X_{k,r}$
$$
(f_{a,b}^{k+p})'(x)\geq \Cr{c:celb} e^{\Cr{kappa:outside}k}e^{\Cr{kappa:trunc} p}\geq e^{\Cl[Kappa-const]{kappa:8}(k+p)},
$$
where $\Cr{kappa:8}=\min(\Cr{kappa:outside},\Cr{kappa:trunc}/2)$.

Here we have used the fact that also for the minimal possible $p$ the factor  $e^{\frac{\Cr{kappa:trunc}}{2}p}$ always
compensates the constant $\Cr{c:celb}$ of Lemma \ref{mane}, and this constant is
independent of $\delta$.

Hence we know that the entire set ${\mathbb T}$ can be written as a
disjoint union of sets $\{Y_j\}_{j=1}^J$ so that for some
$\Cl[Kappa-const]{kappa:10}$ and all $x\in Y_j$
$$
(f_{a,b}^{n_j})'(x)\geq e^{\Cr{kappa:10} n_j}.
$$
We  start with an $x\in Y_{j_0}$. After $n_{j_0}$ steps we will
end up in $Y_{j_1}$ and after another $n_{j_1}$ steps we will end up
in $Y_{j_2}$ etc. The total time will be
$n_{j_0}+n_{j_1}+n_{j_2}+\dots +n_{j_s}$ and
$$
(f_{a,b}^{n_{j_s}+\dots +n_{j_0}})'(x)\geq e^{\Cr{kappa:10}(n_{j_s}+\dots +n_{j_0})},
$$
where
$$
n_{j_s}+\dots+n_{j_0}=\sum_{i=1}^{m} k_in_i.
$$
Let $n_{\text{max}}=\max_{1\leq j \leq J}n_j$ and pick an integer $N$
very large so that
\begin{equation}\label{compensation}
e^{\Cr{kappa:10} N}\cdot d_1^{n_{\text{max}}}\geq e^{\Cl[Kappa-const]{kappa:11} N}.
\end{equation}
Here $d_1=1/B$, where $B=4\geq \max_{x\in{\mathbb T}}|f_{a,b}'(x)|$.

For each point $x$ there is an $n=n(x)=n_{j_0}+n_{j_1}+n_{j_2}+\dots
+n_{j_s}$ so that
$$
N\leq n \leq N+n_{\text{max}}.
$$
We claim that 
$$
(f_{a,b}^N)'(x)\geq e^{\Cr{kappa:11} N}.
$$

This follows since 
$$
(f_{a,b}^N)'(x)= (f_{a,b}^n)'(x)/(f_{a,b}^{n-N})'(f_{a,b}^{N}(x))\geq
e^{\Cr{kappa:10} N}d_1^{n_\text{max}}\geq e^{\Cr{kappa:11} N},
$$
for a suitably $\Cr{kappa:11}$. We conclude that the statement of Proposition \ref{hyper1} holds. \qedwhite

\medskip
We now have all ingredients for the proof of Theorem A.

\medskip
Let $\omega_0$ be an interval as defined in Proposition \ref{MainThmfixedbCE} 
satisfying \eqref{eq:inclusion} and let $\tilde{E}_b$ be the set
defined in this proposition. Let $\hat{E}_b=\hat{E}_b^{\hat{N}}\supset
\tilde{E}_b$ be the set corresponding to the $\hat{N}$:th order construction
of Proposition \ref{hyper1}. $\hat{N}$ is here determined as the
smallest integer satisfying $e^{-\hat{N}}\leq d$ as in the proof of
Proposition \ref{hyper1}. By \eqref{hyperbolicity} it then follows
that the conclusion of Theorem A holds.

\section{Proof of Theorem B}\label{pfthmc}

In this section we are going to prove the last result of the paper.
The methods of its proof will be completely different than the ones
used in the rest of the paper. We will use the term ``countable'' in
the sense ``at most countable.'' For the definitions, see
Introduction.

\begin{proof}[Proof of Theorem~B]
Fix $b<1$. Each tongue is open, so the set $T_b$ is open. Therefore
it is the union of countably many components, each of them an open
interval. Since the points on the boundary of a tongue belong to $TN$,
and the sets $T$ and $TN$ are disjoint, each component is contained in
one tongue.

We claim that the intersection of the closures of two distinct
components $A_1$ and $A_2$ is empty. Suppose it is not and that $a$
belongs to this intersection. Then $(a,b)\in TN$, so it has its type.
This type must be the same as the type of each of the tongues
containing $A_1$ and $A_2$, so those types are the same, that is,
$A_1$ and $A_2$ are contained in the same tongue. If $n$ is the period
of the neutral periodic orbit of $f_{a,b}$, the map $f_{a,b}^n$ has an
interval on which it looks like one of the Cases 1, 2 or 4 of
Lemma~4.1 of~\cite{MR07}. By Theorem~4.1 and Lemma~2.6 of~\cite{MR08},
this cannot be Case~4 (a neutral periodic point repelling from both
sides), and by Lemma~4.2 of~\cite{MR07} it cannot be Case~1 or~2 (a
neutral periodic point repelling from one side). This proves our
claim.

If a parameter $a\in TN_b$ does not belong to a boundary of a
component of $T_b$, then by Lemma~4.2 of~\cite{MR07} the neutral
periodic orbit of $f_{a,b}$ is repelling from both sides (Case~4), so
by Theorem~4.1 and Lemma~2.6 of~\cite{MR08} $a$ is isolated in the set
of elements of $T_b\cup TN_b$ which have type of the same period. This
proves that there are only countably many such values of $a$.

By the claim, the complement of $T_b$ is a closed set without isolated
points. The set $TN_b$ is countable. Therefore $E_b$ (which is the
complement of $T_b$ minus $TN_b$) is dense in the complement of $T_b$.

The second part of the statement follows from the first one and the
fact that each component of $T_b$ is contained in one tongue.
\end{proof}


\begin{thebibliography}{22}

\bibitem{A65}
V.~I.~Arnold, Small denominators, I: Mappings of the Circumference
onto Itself. \emph{Amer. Math. Soc. Translations} \textbf{46},
213--284, (1965).

\bibitem{BC1}
M.~Benedicks and L.~Carleson, On iterations of $1-ax\sp 2$ on $(-1,1)$.
\emph{Ann. of Math. (2)} \textbf{122}, no. 1, 1--25, (1985).

\bibitem{BC2}
M.~Benedicks and L.~Carleson, The dynamics of the H{\'e}non map.
\emph{Ann. of Math. (2)} \textbf{133}, no. 1, 73--169, (1991).

\bibitem{BR09}
M.~Benedicks and A.~Rodrigues, Kneading sequences for double standard
maps. \emph{Fund. Math.} \textbf{206}, 61--75, (2009).

\bibitem{BLvS03}
H.~Bruin, S.~ Luzzatto and S.~van Strien, Decay of correlations in
one-dimensional dynamics. \emph{Ann. Sci. \'{E}cole Norm. Sup. (4)}
\textbf{36}, no. 4, 621--646, (2003).

\bibitem{BRLSvS08}
H.~Bruin, J.~Rivera-Letelier, S.~van Strien and W.~Shen. Large derivatives,
backward contraction and invariant densities for interval maps.
\emph{Invent. Math.} \textbf{172}, no. 3, 509--533, (2008).

\bibitem{CE1980}
P.~Collet and J.-P.~Eckmann, On the abundance of aperiodic behaviour for
maps on the interval. \emph{Comm. Math. Phys.} \textbf{73} 115--160
(1980).

\bibitem {D10}
A.~Dezotti, Connectedness of the Arnold tongues for Double Standard
Maps. \emph{Proc. Amer. Math. Soc.} \textbf{138}, 3569--3583 (2010).

\bibitem{dMvS} Welington de Melo; Sebastian van Strien,
  One-dimensional dynamics. Ergebnisse der Mathematik und ihrer
  Grenzgebiete (3) [Results in Mathematics and Related Areas (3)],
  25. Springer-Verlag, Berlin, 1993  
\bibitem {FG07}
N.~Fagella and A.~Garijo, The Parameter Planes of $\lambda z^m \exp(z)$
for $m \ge 2$. \emph{Comm. Math. Phys.} \textbf{273}, 755--783 (2007).

\bibitem{Ja81}
M.~Jakobson, Absolutely continuous invariant measures for
one-parameter families of one-dimensional maps. \emph{Comm. Math.
  Phys.}, \textbf{81}, 39--88, (1981).

\bibitem{KS69}
K.~Krzy\.zewski and W.~Szlenk, On invariant measures for expanding
differentiable mappings. \emph{Studia Mathematica}, \textbf{33},
83--92 (1969).

\bibitem{LY}
A.~Lasota and J.~A.~Yorke, On the existence of invariant measures for
piecewise monotonic transformations, \emph{Trans. Amer. Math. Soc.}
\textbf{186}, 481--488 (1973).

\bibitem{LevvS2001}
G.~Levin and S.~van~Strien, Bounds for maps of an interval with one
critical point of inflection type. II. \emph{Invent. Math.}
\textbf{141} no. 2, 399--465, (2000).

\bibitem{LS2002}
G.~Levin and G.~\'Swi\c{a}tek, Universality of critical circle covers.
\emph{Commun. Math. Phys.} \textbf{228}, 371--399 (2002).

\bibitem {M85}
R.~Ma\~n\'e, Hyperbolicity, Sinks and Measure in One-Dimensional Dynamics.
\emph{Comm. Math. Phys.} {\bf 100}, 495--524 (1985).

\bibitem {Mis1}
M.~Misiurewicz, Absolutely continuous measures for certain maps of an
interval. \emph{Inst. Hautes Etudes Sci. Publ. Math.} {\bf No. 53},
17--51 (1981).

\bibitem {Mis}
M.~Misiurewicz, Maps of an interval. \emph{Chaotic Behaviour of
  Deterministic Systems}, North-Holland, 565--590 (1983).

\bibitem {MR07}
M.~Misiurewicz and A.~Rodrigues, Double Standard Maps. \emph{Comm. Math.
Phys.} \textbf{273}, 37--65 (2007).

\bibitem {MR08}
M.~Misiurewicz and A.~Rodrigues, On the tip of the tongue. \emph{J. of
Fixed Point Theory Appl.} {\bf 3}, 131--141 (2008).

\bibitem {MR11}
M.~Misiurewicz and A.~Rodrigues, Non-Generic Cusps. \emph{Trans. Amer.
  Math. Soc.} {\bf 363}, 3553--3572 (2011).

\bibitem{SS85}
M.~Shub and D.~Sullivan, Expanding endomorphisms of thee circle
revisited. \emph{Ergodic Theory Dynam. Systems}, \textbf{5}, 285--289
(1985).

\bibitem{NvS} T.~Nowicki, S.~van Strien, Absolutely continuous
  invariant measures for $C^2$ unimodal maps satis fying the
  Collet-Eckmann conditions. \emph{Inventiones math.} {\bf 93},
    619--635 (1988).
\bibitem{ShSu}M.~Shub, D.~Sullivan, Expanding endomorphisms of the
  circle revisited. \emph{Ergodic Theory Dynam. Systems}
  \textbf{5},285--289 (1985).

  \bibitem{vST}
S.~van~Strien, Hyperbolicity and invariant measures for general $C^2$
interval maps satisfying the Misiurewicz condition. \emph{Commun. Math
  Phys.} \textbf{128}, 437--496 (1990).

\bibitem{TTY} Ph.~Thieullen, C.~Tresser and L.-S.~Young, Positive
  Lyapunov exponent for generic one-parameter families of unimodal
  maps. \emph{J. Anal. Math.}  \textbf{64} (1994), 121-172.

\end{thebibliography}
\end{document}